\documentclass{article}

\usepackage{latexsym}
\usepackage{amsmath, amsfonts, amsthm}
\usepackage{epsfig}
\usepackage{stmaryrd}
\usepackage{multicol}
\usepackage{pdfsync}
\usepackage{dsfont}
\usepackage{comment}
\usepackage{algorithmic}
\usepackage{algorithm}
\usepackage{caption}
\usepackage{psfrag}
\usepackage{subfig}
\usepackage{xcolor}
\usepackage{graphics}
\usepackage{setspace}

\usepackage{mathtools}

\DeclarePairedDelimiter\floor{\lfloor}{\rfloor}

\newcommand{\la}{\left \langle}
\newcommand{\ra}{\right\rangle}

\newcommand{\norm}[1]{\left\lVert #1 \right\rVert}

\newtheorem{theorem}{Theorem}[section]
\newtheorem{corollary}[theorem]{Corollary}
\newtheorem{lemma}[theorem]{Lemma}

\theoremstyle{definition}

\newtheorem{assumption}[theorem]{Assumption}

\theoremstyle{remark}
\newtheorem{remark}[theorem]{Remark}

\numberwithin{equation}{section}

\let\OLDthebibliography\thebibliography
\renewcommand\thebibliography[1]{
  \OLDthebibliography{#1}
  \setlength{\parskip}{0pt}
  \setlength{\itemsep}{0pt plus 0.3ex}
}

\title{Mean Field Analysis of Neural Networks: A Law of Large Numbers}

\author{Justin Sirignano\footnote{Department of Industrial \& Systems Engineering, University of Illinois at Urbana Champaign, Urbana, E-mail: jasirign@illinois.edu} \phantom{.}  and Konstantinos Spiliopoulos\footnote{Department of Mathematics and Statistics, Boston University, Boston, E-mail: kspiliop@math.bu.edu}
\thanks{K.S. was partially supported by the National Science Foundation (DMS 1550918)}\\
}

\date{\today}

\begin{document}

\maketitle

\begin{abstract}
Machine learning, and in particular neural network models, have revolutionized fields such as image, text, and speech recognition. Today, many important real-world applications in these areas are driven by neural networks. There are also growing applications in engineering, robotics, medicine, and finance. Despite their immense success in practice, there is limited mathematical understanding of neural networks. This paper illustrates how neural networks can be studied via stochastic analysis, and develops approaches for addressing some of the technical challenges which arise. We analyze one-layer neural networks in the asymptotic regime of simultaneously (A) large network sizes and (B) large numbers of stochastic gradient descent training iterations. We rigorously prove that the empirical distribution of the neural network parameters converges to the solution of a nonlinear partial differential equation. This result can be considered a law of large numbers for neural networks. In addition, a consequence of our analysis is that the trained parameters of the neural network asymptotically become independent, a property which is commonly called ``propagation of chaos".
\end{abstract}

\section{Introduction}

Neural networks have achieved immense practical success over the past decade. Neural networks are nonlinear statistical models whose parameters are estimated from data using stochastic gradient descent. They have been employed as critical components of many important technologies in a variety of industries. This practical success has sparked significant interest in their mathematical analysis. Currently, there is limited mathematical understanding of neural networks. This paper analyzes the asymptotic behavior of neural networks, rigorously proving that the empirical distribution of their parameters converges to the solution of a nonlinear partial differential equation (PDE).

Neural network models have revolutionized fields such as image, text, and speech recognition. They are actively used in a variety of applications. In image recognition, neural networks are able to accurately identify and recognize objects in images using only the raw pixels. Neural networks are used for image recognition in applications such as self-driving cars, image searches on search engines such as Google, and facial recognition for security systems (see \cite{LeCun}, \cite{Goodfellow}, \cite{DriverlessCar}, and \cite{FacialRecognition}). In speech recognition, neural networks are used to develop computer systems that automatically understand human speech (see \cite{LeCun}, \cite{DeepVoice}, \cite{GoogleDuplex}, and \cite{SpeechRecognition3}). Applications include voice control of certain systems in vehicles, transcription (automatically converting human speech to written text), interactive voice response for customer service, and spoken commands for smartphones. In text recognition, neural networks are used to automatically translate text from one language (e.g., English) to another language (e.g., Italian); see \cite{GoogleMachineTranslation} and \cite{MachineTranslation2}. They have also been used for automatically generating summaries of long documents; see \cite{DocumentSummarization} and \cite{DocumentSummarization2}.

In addition, there is growing interest in applying neural networks to engineering, robotics, medicine, and finance. Neural networks are being used in reduced-form models of the Navier-Stokes equation in turbulent conditions (see \cite{Ling1} and \cite{Ling2}). \cite{Robotics1}, \cite{Robotics2}, and \cite{Robotics3} describe applications in robotics. Neural networks have been used to identify cancer \cite{NatureMedicine1} and to model protein folding \cite{NatureMedicine2}. In finance, neural networks have been used to model loan default and prepayment risk \cite{Finance1} and to model high frequency financial data \cite{Finance2}. Neural networks have also been used to solve high-dimensional PDEs in financial applications \cite{Finance3}.

Due to the impact that neural networks have had on practical applications, there is a significant interest in better understanding their mathematical properties. However, the existing literature is relatively limited, with only a few recent papers such as \cite{Telgarsky1}, \cite{Mallat}, and \cite{Telgarsky2}. There also exist classical results regarding the approximation power of neural networks \cite{Barron}, \cite{Hornik1}, and \cite{Hornik2}.

Our result characterizes neural networks with a single hidden layer in the asymptotic regime of large network sizes and large numbers of stochastic gradient descent iterations. We rigorously prove that the empirical distribution of the neural network parameters will weakly converge to a distribution. This distribution satisfies a nonlinear partial differential equation. The proof relies upon weak convergence analysis for interacting particle systems. The result can be considered a ``law of large numbers" for neural networks when both the network size and the number of stochastic gradient descent steps grow to infinity.

Recently, \cite{Mattingly} rigorously established a weak convergence result for a class of machine learning algorithms.  Weak convergence analysis has been widely used in other fields (for example, see \cite{TypicalDefaults}, \cite{LargePortfolio}, \cite{DaiPra1}, \cite{DaiPra2}, \cite{DaiPra3}, \cite{Capponi}, and \cite{Hambly} for a non-exhaustive list). In fact, mean field analysis has been actively used for many years to study biological neural networks and physical systems of interacting particles; see for example \cite{Delarue}, \cite{Inglis}, \cite{Moynot},  \cite{Sompolinsky}, and the references therein.

Upon completion of this work, we became aware of the very recent work of \cite{Montanari} where a related PDE limit result for neural networks is derived; see also the recent work \cite{Rotskoff_VandenEijnden2018}. Our convergence analysis, setup, and assumptions are different. In \cite{Montanari}, it is assumed  that the gradient of the neural network is a priori globally Lipschitz and bounded and under this assumption a similar PDE result as well as certain rates of convergence are established. In our work, we do not assume that the gradient of the neural network is a priori globally Lipschitz or bounded. Often, neural network models (and their gradients) are not globally Lipschitz and not bounded. In this paper, we assume that the data come from a distribution that has its first moments bounded and that the initialization is done according to distributions with certain moment bounds.  Based on this assumption, we rigorously prove relative compactness of the pre-limit measure valued process, identification of its limit, and uniqueness of the limit point in the appropriate space. Our method of proof leverages on weak convergence analysis in an appropriate Skorokhod space for measure-valued processes (similar to the approaches in \cite{Mattingly} and \cite{TypicalDefaults}). In particular, the relative compactness and uniqueness proof addresses the challenge of neural networks not being a priori globally Lipschitz nor globally bounded using the structure of the stochastic gradient descent algorithm.

Consider the one-layer neural network
\begin{eqnarray}
g_{\theta}^N(x) = \frac{1}{N} \sum_{i=1}^N c^i \sigma( w^i \cdot x),\label{Eq:NN}
\end{eqnarray}
where for every $i\in\{1,\cdots, N\}$, $c^i \in \mathbb{R}$ and $x, w^i \in \mathbb{R}^{d}$. For notational convenience we shall interpret $w^i \cdot x=\sum_{j=1}^{d}w^{i,j} x^{j}$ as the standard scalar inner product. The neural network model has parameters $\theta = (c^1, \ldots, c^N, w^1, \ldots, w^N)\in\mathbb{R}^{(1+d)N}$, which must be estimated from data.

The neural network (\ref{Eq:NN}) takes a linear function of the original data, applies an element-wise nonlinearity using the function $\sigma: \mathbb{R} \rightarrow \mathbb{R}$, and then takes another linear function to produce the output. The activation function $\sigma(\cdot)$ is a nonlinear function such as a sigmoid or tanh function. The quantity $\sigma( w^i \cdot x)$ is referred to as the $i$-th ``hidden unit", and the vector $\big{(} \sigma( w^1 \cdot x), \ldots, \sigma( w^N \cdot x) \big{)}$ is called the ``hidden layer". The number of units in the hidden layer is $N$.

The objective function is
\begin{eqnarray}
L^{N}(\theta) = \frac{1}{2}\mathbb{E}_{Y,X} [ ( Y - g_{\theta}^N(X) )^2 ],\label{Eq:ObjFunction}
\end{eqnarray}
where the data $(Y,X)$ is assumed to have a  joint distribution $\pi (dx,dy)$. We shall write $\mathcal{X},\mathcal{Y}$ for the state spaces of $X$ and $Y$, respectively. The parameters $\theta = (c^1, \ldots, c^N, w^1, \ldots, w^N)$ are estimated using stochastic gradient descent:
\begin{eqnarray}
c_{k+1}^i &=& c^i_k + \frac{\alpha}{N} (y_k - g_{\theta_k}^N(x_k) )  \sigma (w^i_k \cdot x_k), \notag \\
w^{i,j}_{k+1} &=& w^{i,j}_k + \frac{\alpha}{N} (y_k - g_{\theta_k}^N(x_k) )  c^i_k \sigma' (w^i_k \cdot x_k) x^{j}_k, \quad j=1,\cdots, d,\label{Eq:SGD}
\end{eqnarray}
where $\alpha$ is the learning rate and $(x_k, y_k) \sim \pi(dx,dy)$. Stochastic gradient descent minimizes (\ref{Eq:ObjFunction}) using a sequence of noisy (but unbiased) gradient descent steps $\nabla_{\theta} [ ( y_k - g_{\theta_k}^N(x_k) )^2 ]$. Note that typically $\nabla_{\theta}[  ( y - g_{\theta}^N(x) )^2]$ is not a priori globally Lipschitz nor globally bounded as a function of $\theta$. Stochastic gradient descent typically converges more rapidly than gradient descent for large datasets. For this reason, stochastic gradient descent is widely used in machine learning.

Define the empirical measure
\begin{eqnarray}
\nu^N_k(dc, dw) = \frac{1}{N} \sum_{i=1}^N \delta_{c_k^i, w_k^i}(dc, dw).
\end{eqnarray}

The neural network's output can be re-written in terms of the empirical measure:
\begin{eqnarray}
g_{\theta_k}^N(x)  = \la c \sigma(w \cdot x),  \nu^N_k \ra.
\end{eqnarray}
$\la f, h \ra$ denotes the inner product of $f$ and $h$. For example, $ \la c \sigma(w \cdot x),  \nu^N_k \ra = \int c \sigma(w \cdot x) \nu^N_k(dc, dw)  $.

The scaled empirical measure is
\begin{eqnarray}
\mu^N_t = \nu^N_{\floor*{N t} }.
\end{eqnarray}

At any time $t$, the scaled empirical measure $\mu^N_t$ is a random element of $D_{E}([0,T])=D([0,T];E)$\footnote{$D_S([0,T])$ is the set of maps from $[0,T]$ into $S$ which are right-continuous and which have left-hand limits.} with $E = \mathcal{M}(\mathbb{R}^{1+d})$. We study the convergence in distribution of $\mu^N_t$ in the Skorokhod space $D_E([0,T])$.

Our main results are stated below. Theorem \ref{TheoremLLN} (and the associated Remark \ref{R:ProbConvergence}) is a law of large numbers describing the distribution of the trained parameters when $N$ is large. Theorem \ref{TheoremChaos} describes the behavior of individual parameters when $N$ is large. Theorem \ref{TheoremChaos} is a ``propagation of chaos" result. Section \ref{Insights} presents several insights provided by these asymptotic results.

We shall work on a filtered probability space $(\Omega,\mathcal{F},\mathbb{P})$  on which all the random variables are defined. The probability space is equipped with a filtration $\mathfrak{F}_t$ that is right continuous and $\mathfrak{F}_0$ contains all $\mathbb{P}$-negligible sets.

At this point, let us recall the definition of chaoticity. Let $q$ be a probability measure on a Polish space $\mathcal{Z}$ and, for $N\in\mathbb{N}$, let $\mathbb{Q}^{N}$ be a symmetric probability measure on the product space $\mathcal{Z}^{N}$. Then $(\mathbb{Q}^{N})_{N\in\mathbb{N}}$ is called $q-$chaotic if, for every $k\in\mathbb{N}$, the joint distribution law of the first $k$ marginals of $\mathbb{Q}^{N}$ converge weakly to the product measure $\otimes^{k}q$.

We impose the following assumption.
\begin{assumption} \label{A:Assumption1} We have that
\begin{itemize}
\item The activation function $\sigma\in C^{2}_{b}(\mathbb{R})$, i.e. $\sigma$ is twice continuously differentiable and bounded.
\item The sequence of data samples $(x_k, y_k)$ is i.i.d. from a probability distributed $\pi(dx,dy)$ such that $\mathbb{E}\parallel x_{k}\parallel^{4}+\mathbb{E}|y_{k}|^{4}$ is bounded.
\item The randomly initialized parameters $(c_0^i, w_0^i)$ are i.i.d. with a distribution $\bar \mu_0$ such that $\mathbb{E}[ \exp \big{(} q | c_0^i | \big{)} ] < C$ for some $0 < q < \infty$ and $\mathbb{E}[ \parallel w_0^i \parallel^4 ] < C$.

\end{itemize}
\end{assumption}

Notice that initial distributions on $c^{i}$ with compact support or exponential tails satisfy the condition on the moment generating function of $|c_0^i |$ in the assumption above. Under Assumption \ref{A:Assumption1}, the initial empirical measure satisfies $\mu_0^N \overset{d} \rightarrow \bar \mu_0$ as $N \rightarrow \infty$. In addition, due to our assumption on the distribution of the $(x_k, y_k)$ data and of the initialization $(c_0^i, w_0^i)_{i=1}^N$, the joint distribution of $ ( c^i_k, w^i_k)_{i=1}^N \in  ( \mathbb{R}^{1+d} )^{\otimes N}$ is exchangeable and, consequently, $\nu_k^N$ is a Markov chain in the space of probability measures on $E$.

\begin{theorem} \label{TheoremLLN}
 Assume Assumption \ref{A:Assumption1}. The scaled empirical measure $\mu^N_t$ converges in distribution to a limit measure $\bar \mu_t$ with values  in $D_E([0,T])$ as $N \rightarrow \infty$. For every $f\in C^{2}_{b}(\mathbb{R}^{1+d})$, $\bar \mu$ is the unique deterministic solution of the  measure evolution equation
\begin{eqnarray}
\la f, \bar \mu_t \ra  &=& \la f, \bar \mu_0 \ra + \int_0^t   \bigg{(} \int_{\mathcal{X}\times\mathcal{Y}}   \alpha \big{(} y -  \la c' \sigma(w'\cdot x),  \bar \mu_s \ra \big{)} \la \sigma(w \cdot x) \partial_c f, \bar \mu_s \ra   \pi(dx,dy)\bigg{)} ds \notag \\
& &+ \int_0^t \bigg{(} \int_{\mathcal{X}\times\mathcal{Y}}  \alpha \big{(} y -  \la c' \sigma(w' \cdot x), \bar  \mu_s \ra \big{)} \la c \sigma'(w \cdot x) x \cdot \nabla_w f, \bar \mu_s \ra \pi(dx, dy) \bigg{)}   ds \notag \\
&=& \la f, \bar \mu_0 \ra + \int_0^t   \bigg{(} \int_{\mathcal{X}\times\mathcal{Y}}   \alpha \big{(} y -  \la c' \sigma(w'\cdot x),  \bar \mu_s \ra \big{)} \la \nabla (c\sigma(w \cdot x)) \cdot \nabla f, \bar \mu_s \ra   \pi(dx,dy)\bigg{)} ds,
\label{EvolutionEquationIntroduction}
\end{eqnarray}
where $\nabla f=(\partial_{c}f,\nabla_{w}f)$.
\end{theorem}

\begin{remark}\label{R:ProbConvergence}
Since weak convergence to a constant implies convergence in probability, Theorem \ref{TheoremLLN} leads to the stronger result of convergence in probability
\begin{equation*}\label{E:mulimit}
  \lim_{N\to \infty}\mathbb{P}\left\{ d_{E}(\mu^N,\bar \mu)\ge \delta\right\} = 0
  \end{equation*}
for every $\delta>0$ and where $d_E$ is the metric for $D_E([0,T])$.
\end{remark}
\begin{corollary} \label{CorollaryLLN}
 Assume Assumption \ref{A:Assumption1}.  Suppose that $\bar \mu_0$ admits a density $p_0(c,w)$ and there exists a unique solution to the nonlinear partial differential equation
  \begin{eqnarray}
\frac{ \partial p(t, c, w) }{ \partial t}  &=& - \alpha \int_{\mathcal{X}\times\mathcal{Y}}   \bigg{(} \big{(} y -  \la c' \sigma(w' \cdot x), p(t,c', w') \ra \big{)} \frac{\partial}{\partial c} \big{[} \sigma(w \cdot x) p(t,c,w) \big{]} \bigg{)}\pi(dx,dy)  \notag \\
&-& \alpha \int_{\mathcal{X}\times\mathcal{Y}}   \bigg{(} \big{(} y -  \la c' \sigma(w' \cdot x),  p(t,c',w') \ra \big{)} x \cdot \nabla_{w} \big{[}  c \sigma'(w \cdot x)  p(t,c,w) \big{]} \bigg{)}\pi(dx,dy), \notag \\
p(0,c,w) &=& p_0(c,w), \notag
\label{fpPDE}
\end{eqnarray}
such that $p(t, c, w)$ vanishes as $|c|, \parallel w\parallel\rightarrow\infty$. Then, we have that the solution to the measure evolution equation (\ref{EvolutionEquationIntroduction}) is such that
\begin{eqnarray*}
\bar \mu_t(dc, dw) = p(t, c,w) dc dw. \notag
\end{eqnarray*}
\end{corollary}

\begin{remark}\label{R:GradientFlow}
Notice that by setting $\theta=(c,w)$, the partial differential equation for $p(t,\theta)$ in Corollary \ref{CorollaryLLN} can be written as
\begin{align}
\frac{ \partial p(t, \theta) }{ \partial t}  &=  -\alpha \textrm{div}_{\theta}\left(p(t,\theta)\nabla_{\theta}v(\theta,p(t,\cdot))\right) \notag \\
p(0,\theta) &= p_0(\theta), \label{GradientFlow}
\end{align}
where $\textrm{div}_{\theta}$ is the divergence operator with respect to the variable $\theta$ and $v(\theta,p(t,\cdot))$ is defined as
\begin{align}
v(\theta,p(t,\cdot))&=\int_{\mathcal{X}\times\mathcal{Y}}
\left( \left( y -  \int_{\mathbb{R}^{1+d}} c' \sigma(w' \cdot x) p(t,c', w') dc'dw' \right) c \sigma(w \cdot x)   \right)\pi(dx,dy).\nonumber
\end{align}

In addition, Theorem \ref{TheoremLLN} and Corollary \ref{CorollaryLLN} imply that the objective function $L^{N}(\theta)$ from (\ref{Eq:ObjFunction}) satisfies
\begin{align}
\lim_{N\rightarrow\infty}L^{N}(\theta_{\floor*{Nt}})&=\bar{L}(p(t,\cdot))
=\frac{1}{2}\int_{\mathcal{X}\times\mathcal{Y}}\left(y-\bar{g}(x,p(t,\cdot))\right)^{2}\pi(dx,dy),\quad \textrm{where,}\label{Eq:LimitObjectiveFcn}\\
\bar{g}(x,p(t,\cdot))&=\int_{\mathbb{R}^{1+d}} c\sigma(w\cdot x) p(t,c,w)dcdw.\nonumber
\end{align}

As it is well known in the literature, see for example \cite{Ambrosio2008,Carrillo2003,Jordan1998}, the PDE (\ref{GradientFlow}) is a gradient flow for the limiting objective function (\ref{Eq:LimitObjectiveFcn}) in the space of probability measures on $\mathbb{R}^{1+d}$ endowed with the Wasserstein metric. As it is also extensively discussed in \cite{Ambrosio2008}, but also  pointed out in the recent works by \cite{Chizat2018,Montanari,Szpruch2019}, this means that the trajectory's  $t\mapsto p(t,\cdot)$ goal is to minimize the limit objective function $\bar{L}(p)$ as defined by (\ref{Eq:LimitObjectiveFcn}). For more details on the related optimal transportation theory we refer the interested reader to Lemma 10.3.16 in \cite{Ambrosio2008} for general results, and to  Proposition 4.2 in \cite{Chizat2018} and to Corollary 3.3 in \cite{Szpruch2019} for a more specialized discussion in the special case of limiting problems of the type (\ref{GradientFlow})-(\ref{Eq:LimitObjectiveFcn}).
\end{remark}

In Theorem \ref{TheoremChaos} we prove that the neural network has the  ``propagation of chaos" property.
\begin{theorem} \label{TheoremChaos}
 Assume Assumption \ref{A:Assumption1}. Consider $T<\infty$ and let $t\in(0,T]$. Define the probability measure $\rho^N_t \in \mathcal{M} (\mathbb{R}^{(1+d)N})$ where
\begin{eqnarray*}
\rho^N_t(dx^1, \ldots, dx^N) = \mathbb{P}[ (c_{\floor*{N t}}^1, w_{\floor*{N t}}^1) \in dx^1, \ldots, (c_{\floor*{N t}}^N, w_{\floor*{N t}}^N) \in dx^N ].
\end{eqnarray*}
Then, the sequence of probability measures $\rho^N_\cdot$ is $\bar \mu_\cdot$-chaotic. That is, for $k\in\mathbb{N}$
\begin{eqnarray}
\lim_{N \rightarrow \infty } \la   f_1(x^1) \times \cdots \times f_k(x^k), \rho_\cdot^N(dx^1, \ldots, dx^N) \ra = \prod_{i=1}^k \la  f_i, \bar \mu_{\cdot} \ra, \phantom{....} \forall f_1, \ldots, f_k \in C^{2}_{b}(\mathbb{R}^{1+d}).
\label{PropChaos}
\end{eqnarray}

\end{theorem}

\subsection{Insights from Law of Large Numbers and Numerical Studies} \label{Insights}
The law of large numbers (\ref{EvolutionEquationIntroduction}) suggests several interesting characteristics of trained neural networks (at least in the setting studied in this paper).
\begin{itemize}
\item As $N \rightarrow \infty$, the neural network converges (in probability) to a deterministic model. This is despite the fact that the neural network is randomly initialized and it is trained on a random sequence of data samples via stochastic gradient descent.
\item The learning rate $\alpha$ was assumed to be constant and to not decay with time.  However, notice that the hidden layer has been normalized by $1/N$ and it is this normalization by $1/N$ in the hidden layer that replaces the role of the learning rate decay, enabling convergence. 
\item As it also discussed in Remark \ref{R:GradientFlow}, the PDE (\ref{GradientFlow}) is a gradient flow for the limiting objective function (\ref{Eq:LimitObjectiveFcn}) in the space of probability measures on $\mathbb{R}^{1+d}$ endowed with the Wasserstein metric. Hence, the limiting law of large numbers  goal is to minimize the limit objective function $\bar{L}(p)$ as defined by (\ref{Eq:LimitObjectiveFcn}).
\item The propagation of chaos result (\ref{PropChaos}) indicates that, as $N \rightarrow \infty$, the dynamics of the weights $(c^i_k, w^i_k)$ will become independent of the dynamics of the weights $(c^j_k, w^j_k)$ for any $i \neq j$. Note that the dynamics $(c^i_k, w^i_k)$ are still random due to the random initialization. However, the dynamics of the $i$-th set of weights will be uncorrelated with the dynamics of the $j$-th set of weights in the limit as $N\rightarrow\infty$.
\end{itemize}

In order to illustrate some aspects of the theoretical results of this paper, we performed the following numerical study.

Figure \ref{ConvergenceFigure} displays the convergence of the distribution of the parameters in a trained neural network as the number of hidden units $N \rightarrow \infty$. The neural network has a single hidden layer followed by a softmax function. Figure \ref{ConvergenceFigure} reports the distribution of the parameters connecting the hidden layer to the softmax function. The distributions are presented as histograms. The neural network is trained on the MNIST dataset, which is a standard image dataset in machine learning \cite{MNIST}. The dataset includes $60,000$ images of handwritten numbers $\{ 0,1,2, \ldots, 9 \}$.  The neural network is trained to identify the handwritten numbers using only the image pixels as an input (i.e., it learns to recognize images as a human would). In the MNIST dataset, each image has $784$ pixels. A pixel takes values in $\{0, 1, \ldots, 255 \}$.\footnote{The pixel values are normalized to $[0,1]$ for the purposes of training the neural network.} Neural networks can achieve 98-99\% out-of-sample accuracy on the MNIST dataset.

Figure \ref{ConvergenceFigure} shows that the distribution of parameters converges to a fixed distribution as $N \rightarrow \infty$. This can be seen by the fact that the distributions for $N = 10,000$, $N = 100,000$, and $N = 250,000$ are nearly identical. A priori it is unclear if the distribution of neural network parameters should converge as $N \rightarrow \infty$. Our theory and numerical results confirm that this is indeed the case. Indeed, as $N$ gets large, we see that the empirical distribution of the parameters connecting the hidden layer to the softmax function converges to a specific deterministic distribution.
\begin{figure}[ht!]
\begin{center}
\includegraphics[width=.47\textwidth]{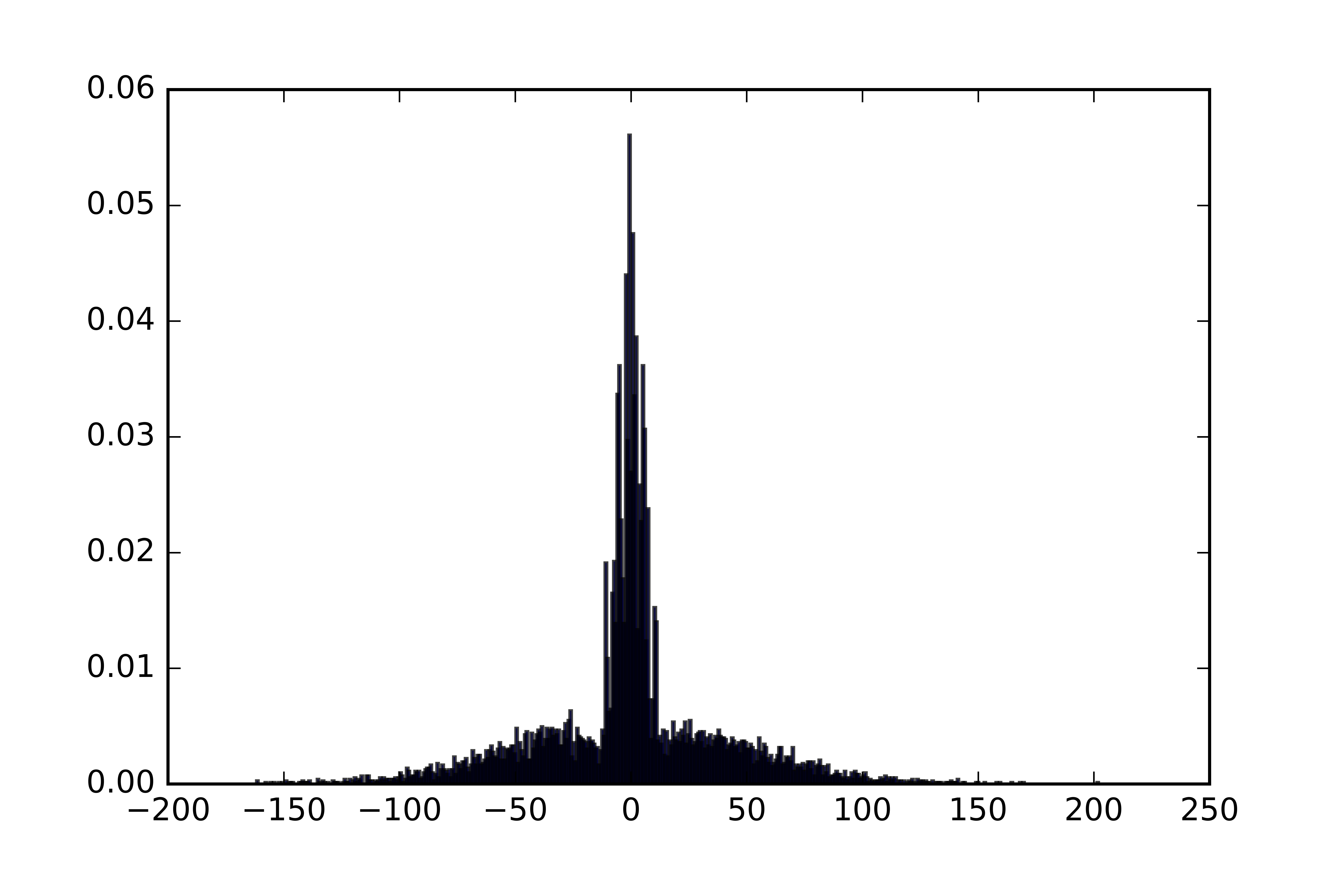}
\includegraphics[width=.47\textwidth]{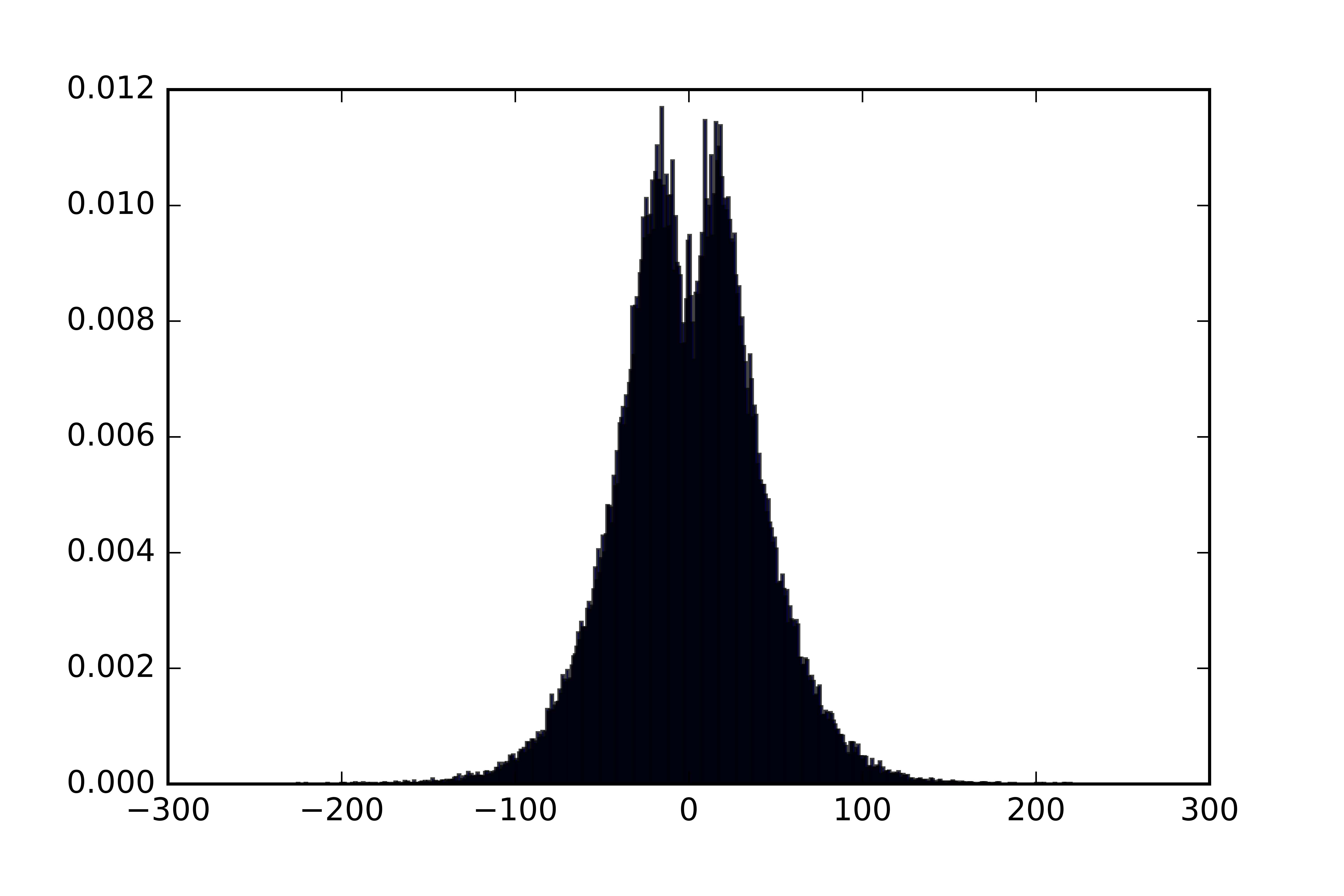}
\includegraphics[width=.47\textwidth]{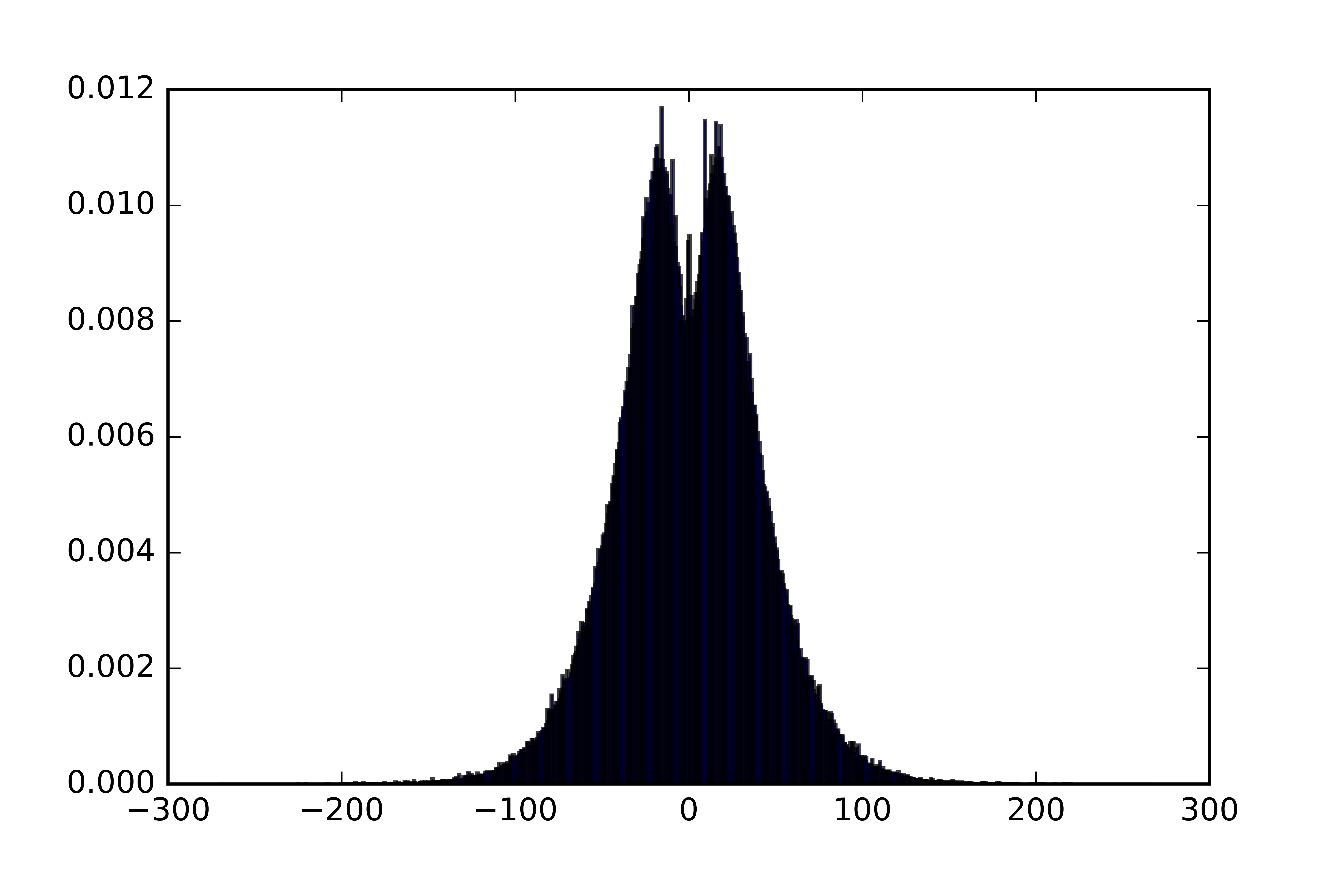}
\includegraphics[width=.47\textwidth]{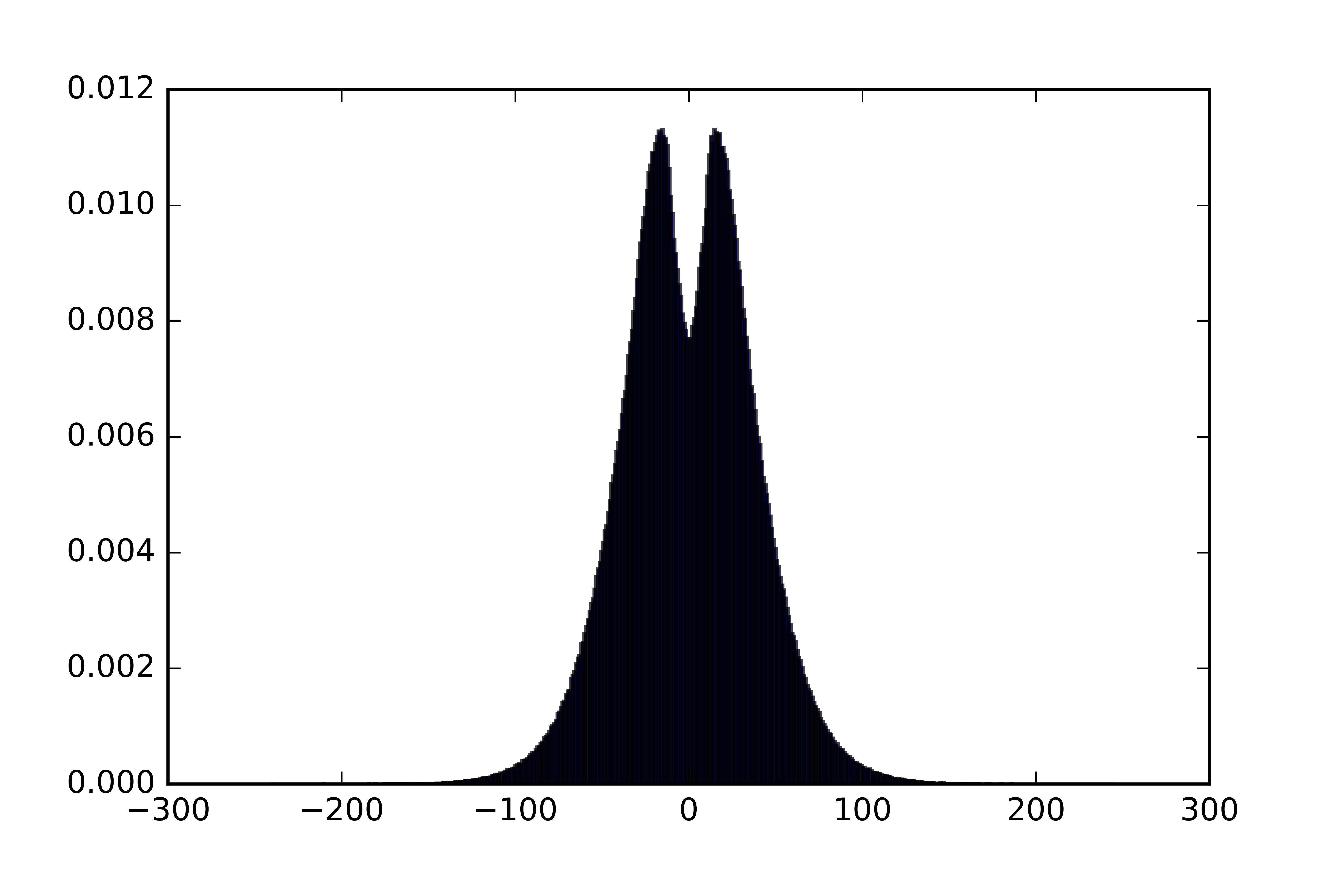}
\end{center}
\caption{Distribution of parameters for a neural network trained on MNIST dataset. Clockwise: $N = 1,000$, $N = 10,000$, $N = 100,000$, and $N = 250,000$ hidden units.}
\label{ConvergenceFigure}
\end{figure}

\subsection{Overview of the Proof } \label{OverviewOfProof}
The rest of the paper is organized as follows. Section \ref{RelativeCompactness} proves relative compactness of the family $\{\mu^{N}\}_{N\in\mathbb{N}}$. Section \ref{Identification} identifies the limit point of any convergent subsequence. The limit point must satisfy the measure evolution equation (\ref{EvolutionEquationIntroduction}). Section \ref{Uniqueness} proves uniqueness of the evolution equation (\ref{EvolutionEquationIntroduction}) via a fixed point argument. Then, by Prokhorov's Theorem, these results prove that the sequence of probability measures $\pi^N$ of the processes $\mu^N$ weakly converge to $\pi$, the probability measure of the process $\bar \mu$ satisfying equation (\ref{EvolutionEquationIntroduction}). These results are collected together in Section \ref{FinalProof} to prove Theorem \ref{TheoremLLN}, Corollary \ref{CorollaryLLN}, and Theorem \ref{TheoremChaos}. We conclude with a discussion of our results in Section \ref{Conclusion}.

\section{Relative Compactness} \label{RelativeCompactness}
We now prove relative compactness of the family $\{\mu^{N}\}_{N\in\mathbb{N}}$ in $D_E([0,T])$ where $E = \mathcal{M} ( \mathbb{R}^{1+d})$.  It is sufficient to show compact containment and regularity of the $\mu^{N}$'s (see for example Chapter 3 of \cite{EthierAndKurtz}). We start with a crucial a-priori bound for the SGD iterates as given by (\ref{Eq:SGD}), Lemma \ref{L:AprioriBound}.
\begin{lemma}\label{L:AprioriBound}
Consider the system (\ref{Eq:SGD}). Then, for $k\leq TN$ and uniformly in $i\in\mathbb{N}$, there exists a constant $C<\infty$ such that
\begin{align*}
\mathbb{E}\left[|c^i_k|+\parallel w^i_k \parallel \right] \leq C.
\end{align*}
In particular, we have
\begin{align*}
\sup_{N\in\mathbb{N}, k/N\leq T}\frac{1}{N}\sum_{i=1}^{N}\mathbb{E}\left[|c^i_k|+\parallel w^i_k \parallel \right] \leq C.
\end{align*}
\end{lemma}

For the purposes of presentation, the proof of Lemma \ref{L:AprioriBound} will be given at the end of this section. First, we prove compact containment for  the measure-valued process $\{\mu_t^N, t\in[0,T]\}_{N\in\mathbb{N}}$.
\begin{lemma}\label{L:CompactContainment}
For each $\eta > 0$, there is a compact subset $\mathcal{K}$ of E such that
\begin{eqnarray*}
\sup_{N \in \mathbb{N}, 0 \leq t \leq T} \mathbb{P}[ \mu_t^N \notin \mathcal{K} ] < \eta.
\end{eqnarray*}
\end{lemma}
\begin{proof}
For each $L>0$, define $K_L=[0,L]^{1+d}$.  Then, we have that $K_L$ is a compact  subset of $\mathbb{R}^{1+d}$, and for each $t\geq 0$ and $N\in \mathbb{N}$,
\begin{equation*}
\mathbb{E}\left[\mu^N_t(\mathbb{R}^{1+d}\setminus K_L)\right] = \frac{1}{N}\sum_{i=1}^N \mathbb{P}\left[ |c^i_{\floor*{N t}}|+\parallel w^i_{\floor*{N t}} \parallel \geq L\right] \leq \frac{C}{L}.
\end{equation*}
where the constant $C<\infty$ is from Lemma \ref{L:AprioriBound}. The rest of the proof is standard now (see for example Lemma 6.1 of \cite{TypicalDefaults}). We define the compact subsets of $E = \mathcal{M} ( \mathbb{R}^{1+d})$
\begin{equation*}
\hat{K}_L  = \overline{\left\{ \nu:\, \nu(\mathbb{R}^{1+d}\setminus K_{(L+j)^2}) < \frac{1}{\sqrt{L+j}} \textrm{ for all } j\in \mathbb{N}\right\}}
\end{equation*}
and we observe that
\begin{align*}
 \mathbb{P}\left\{ \mu^N_t\not \in \hat{K}_L\right] &\leq \sum_{j=1}^\infty \mathbb{P}\left[ \mu^N_t(\mathbb{R}^{1+d}\setminus K_{(L+j)^2} )> \frac{1}{\sqrt{L+j}}\right]
\leq \sum_{j=1}^\infty \frac{\mathbb{E}[\mu^N_t(\mathbb{R}^{1+d}\setminus K_{(L+j)^2})]}{1/\sqrt{L+j}}\\
&\le \sum_{j=1}^\infty \frac{C}{(L+j)^2/\sqrt{L+j}}
\le \sum_{j=1}^\infty \frac{C}{(L+j)^{3/2}}.
\end{align*}
Given now that $\lim_{L\to \infty}\sum_{j=1}^\infty\frac{C}{(L+j)^{3/2}} =0$, the proof of the lemma is concluded.
\end{proof}

We now establish regularity of the $\mu^{N}$'s. Define the function $q(z_{1},z_{2})=\min\{|z_{1}-z_{2}|,1\}$ where $z_{1},z_{2} \in \mathbb{R}$.
\begin{lemma}\label{L:regularity}
Let  $f \in C^{2}_{b}(\mathbb{R}^{1+d})$. For any $p \in (0,1)$, there is a constant $C<\infty$ such that for $0\leq u\leq \delta$,  $0\leq v\leq \delta\wedge t$, $t\in[0,T]$,
\begin{equation*}
 \mathbb{E}\left[q(\left< f,\mu^N_{t+u}\right>,\left< f,\mu^N_t\right>)q(\left< f,\mu^N_t\right>,\left< f,\mu^N_{t-v}\right>)\big| \mathcal{F}^N_t\right]  \le  C \delta^{p} + \frac{C}{N}.
\end{equation*}
%
%
\end{lemma}
\begin{proof}
We start by noticing that a Taylor expansion gives for $0\leq s\leq t\leq T$ \begin{eqnarray}
| \la f , \mu^N_{t} \ra -  \la f , \mu^N_{s} \ra  | &=& | \la f , \mu^N_{t} \ra -  \la f , \mu^N_{s} \ra  | = | \la f,   \nu^N_{\floor*{N t} } \ra -  \la f,   \nu^N_{\floor*{N s} } \ra | \notag \\
&\leq& \frac{1}{N} \sum_{i=1}^N | f( c^i_{\floor*{N t}}, w^i_{\floor*{N t}}) - f(c^i_{\floor*{N s}}, w^i_{\floor*{N s}}) | \notag \\
&\leq& \frac{1}{N} \sum_{i=1}^N  | \partial_c f (  \bar c^i_{\floor*{N t}}, \bar w^i_{\floor*{N t}} ) | |  c^i_{\floor*{N t}} -  c^i_{\floor*{N s}} | \notag \\
&+& \frac{1}{N} \sum_{i=1}^N  \parallel \nabla_w f (  \bar c^i_{\floor*{N t}}, \bar w^i_{\floor*{N t}} ) \parallel \parallel  w^i_{\floor*{N t}} -  w^i_{\floor*{N s}} \parallel,
\label{Regularity1}
\end{eqnarray}
for points $\bar c^{i}, \bar w^{i}$ in the segments connecting $c^i_{\floor*{N s}}$ with $c^i_{\floor*{N t}}$ and  $w^i_{\floor*{N s}}$ with $w^i_{\floor*{N t}}$, respectively.

Let's now establish a bound on $|  c^i_{\floor*{N t}} -  c^i_{\floor*{N s}} |$ for $s < t \leq T$. Let $0 < p < 1$.
\begin{eqnarray}
 \mathbb{E}|  c^i_{\floor*{N t}} -  c^i_{\floor*{N s}} | &=& \mathbb{E} | \sum_{k = \floor*{N s } }^{\floor*{N t}-1} ( c_{k+1} - c_k  ) | \notag \\
&\leq&  \mathbb{E} \sum_{k = \floor*{N s } }^{\floor*{N t}-1}  | \alpha (y_k - g_{\theta_k}^N(x_k) ) \frac{1}{N} \sigma (w^i_k \cdot x_k) | \notag \\
&\leq& \frac{1}{N} \sum_{k = \floor*{N s } }^{\floor*{N t}-1} C  \leq C (t -s ) + \frac{C}{N} \notag \\
&\leq& C (t - s)^{p} \mathbf{1}_{ t -s < 1 }  + C (t -s )^{p} T^{1/p} \mathbf{1}_{t -s \geq 1}  + \frac{C}{N} \notag \\
&\leq& C (t - s)^{p} + \frac{C}{N},\notag
\end{eqnarray}
where Assumption \ref{A:Assumption1} was used. Let's now establish a bound on $\parallel  w^i_{\floor*{N t}} -  w^i_{\floor*{N s}} \parallel$ for $s < t \leq T$. Making use of the uniform bounds established in Lemma \ref{L:AprioriBound}, we obtain similarly to the previous bound
\begin{eqnarray*}
  \mathbb{E} \parallel w^i_{\floor*{N t}} -  w^i_{\floor*{N s}} \parallel &=&  \mathbb{E}\parallel \sum_{k = \floor*{N s } }^{\floor*{N t}-1} ( w_{k+1} - w_k  ) \parallel \notag \\
&\leq&   \mathbb{E} \sum_{k = \floor*{N s } }^{\floor*{N t}-1} \parallel    \alpha (y_k - g_{\theta_k}^N(x_k) ) \frac{1}{N} c^i_k \sigma' (w^i_k \cdot x_k) x_k \parallel \notag \\
&\leq&  \frac{1}{N} \sum_{k = \floor*{N s } }^{\floor*{N t}-1} C  \notag \\
& \leq & C (t -s ) + \frac{C}{N} \leq C (t - s)^{p} + \frac{C}{N} .
\end{eqnarray*}

Now, we return to equation (\ref{Regularity1}). By Lemma \ref{L:AprioriBound}, the quantities $( \bar c^i_{\floor*{N t}}, \bar w^i_{\floor*{N t}} )$ are bounded in expectation for $0 < s < t \leq T$.  Therefore, for $0 < s < t \leq T$,
\begin{eqnarray*}
\mathbb{E}\left[| \la f , \mu^N_{t} \ra -  \la f , \mu^N_{s} \ra  | \big| \mathcal{F}^N_s \right] \leq C (t - s)^{p} + \frac{C}{N}.
\end{eqnarray*}
where $C<\infty$ is some unimportant constant. Then, the statement of the Lemma follows.
\end{proof}

We can now prove the required relative compactness of the sequence $\{\mu^{N}\}_{N\in\mathbb{N}}$. This implies that every subsequence $\mu^N$'s has a convergent sub-subsequence.

\begin{lemma}\label{L:RelativeCompactness}
The sequence of probability measures $\{\mu^N\}_{N\in\mathbb{N}}$ is relatively compact in $D_{E}([0,T])$.
\end{lemma}

\begin{proof}
Given Lemmas \ref{L:CompactContainment} and \ref{L:regularity}, Theorem 8.6 of Chapter 3 of \cite{EthierAndKurtz}, gives the statement of the lemma. (See also Remark 8.7 B of Chapter 3 of \cite{EthierAndKurtz} regarding replacing $\sup_N$ with $\lim_N$ in the regularity condition B of Theorem 8.6.)
\end{proof}

We conclude this section with the proof of the a-priori bound of Lemma \ref{L:AprioriBound}.
\begin{proof}[Proof of Lemma \ref{L:AprioriBound}]
We start by establishing some useful a-priori bounds on $c_k^i$ and $w_k^i$. The unimportant finite constant $C<\infty$ may change from line to line. We first observe that
\begin{eqnarray*}
| c_{k+1}^i  | &\leq&   | c_{k}^i | + \alpha \left| y_k - g_{\theta_k}^N(x_k)  \right| \frac{1}{N} | \sigma (w^i_k \cdot x_k) | \notag \\
&\leq&  | c_{k}^i | + \frac{ \alpha C  | y_k |  }{N} +  \frac{C}{N^2} \sum_{i=1}^N | c_k^i |,
\end{eqnarray*}
where to derive the  last line we used the definition of $g_{\theta_k}^N(x)$ via (\ref{Eq:NN}) and the uniform boundedness assumption on $\sigma$.  Then, we subsequently obtain that
\begin{eqnarray*}
| c_{k}^i | &=& | c_{0}^i |  + \sum_{j = 1}^k [ | c_{j}^i | - | c_{j-1}^i | ] \notag \\
&\leq& | c_{0}^i |  + \sum_{j=1}^k  \frac{ \alpha C  | y_{j-1} |  }{N}  +  \frac{C}{N^2} \sum_{j=1}^k \sum_{i=1}^N | c_{j-1}^i | .
\end{eqnarray*}
This implies that
\begin{eqnarray*}
\frac{1}{N} \sum_{i=1}^N | c_{k}^i | &\leq& \frac{1}{N} \sum_{i=1}^N | c_{0}^i |  + \sum_{j=1}^k  \frac{ \alpha C  | y_{j-1} |  }{N}  +  \frac{C}{N^2} \sum_{j=1}^k \sum_{i=1}^N | c_{j-1}^i |,
\end{eqnarray*}
Let us now define $m_{k}^{N}=\frac{1}{N} \sum_{i=1}^N | c_{k}^i |$ and $b_{k}^{N}=\frac{1}{N} \sum_{i=1}^N | c_{0}^i |  + \sum_{j=1}^k  \frac{ \alpha C  | y_{j-1} |  }{N}$. Then we have
\begin{eqnarray*}
m_{k}^{N} &\leq& b_{k}^{N}  +  \frac{C}{N} \sum_{j=1}^k m_{j}^{N},
\end{eqnarray*}
which by the discrete Gronwall lemma gives the bound
\begin{eqnarray*}
m_{k}^{N} &\leq& b_{k}^{N}  +  \frac{C}{N} \sum_{j=1}^k b_{j}^{N}e^{C\frac{j-i}{N}}\leq b_{k}^{N}  +  \frac{C}{N} \sum_{j=1}^k b_{j}^{N}
\end{eqnarray*}
for a possibly different constant that may depend on $T$, where the relation $k/N\leq T$ was used in the last step. Going back now to the bound for $c_{k}^{i}$ we obtain
\begin{eqnarray*}
| c_{k}^i | &\leq& | c_{0}^i |  + b_{k}^{N}  +  \frac{C}{N} \sum_{j=1}^k m_{j}^{N} .
\end{eqnarray*}
Raising this to power $1\leq p\leq 4$, we have for a constant $C_{p}$ that may depend on $p$
\begin{eqnarray*}
| c_{k}^i |^{p} &\leq& C_{p}\left[| c_{0}^i |^{p}  + |b_{k}^{N}|^{p}  +  \frac{1}{N^{p}} \left|\sum_{j=1}^k b_{j}^{N}\right|^{p} +  \frac{1}{N^{2p}} \left|\sum_{j=1}^k \sum_{i=1}^{j} b_{i}^{N}\right|^{p} \right].
\end{eqnarray*}

Let us bound now each of the terms on the right hand side of the last display. We have for some constant $C_{p}<\infty$ that may change from line to line
\begin{align}
|b_{k}^{N}|^{p}&\leq C_{p}\left[\frac{1}{N}\sum_{i=1}^{N}|c_{0}^{i}|^{p}+\frac{k^{p-1}}{N^{p}}\sum_{j=1}^{k}|y_{j-1}|^{p}\right]\nonumber\\
\frac{1}{N^{p}} \left|\sum_{j=1}^k b_{j}^{N}\right|^{p}&\leq \frac{k^{p-1}}{N^{p}} \sum_{j=1}^k \left|b_{j}^{N}\right|^{p}\leq  C_{p}\frac{k^{p-1}}{N^{p}} \sum_{j=1}^k \left[\frac{1}{N}\sum_{i=1}^{N}|c_{0}^{i}|^{p}+\frac{j^{p-1}}{N^{p}}\sum_{i=1}^{j}|y_{i-1}|^{p}\right]\nonumber\\
\frac{1}{N^{2p}} \left|\sum_{j=1}^k \sum_{i=1}^{j} b_{i}^{N}\right|^{p} &\leq \frac{k^{p-1}}{N^{p}} \sum_{j=1}^k \frac{j^{p-1}}{N^{p}}\sum_{i=1}^{j} \left|b_{i}^{N}\right|^{p} \leq C_{p} \frac{k^{p-1}}{N^{p}} \sum_{j=1}^k \frac{j^{p-1}}{N^{p}}\sum_{i=1}^{j} \left[\frac{1}{N}\sum_{i=1}^{N}|c_{0}^{i}|^{p}+\frac{i^{p-1}}{N^{p}}\sum_{\lambda=1}^{i}|y_{\lambda-1}|^{p}\right]\nonumber
\end{align}

Plugging those bounds now in the previous bound for $| c_{k}^i |^{p}$, using the assumption $E|y_{i}|^{p}\leq C<\infty$ for all $i$ and all $p\in[1,4]$ and that $k/N\leq T$, we obtain that for all $i\in\mathbb{N}$ and all $k$ such that $k/N\leq T$, the bound
\begin{align}
\mathbb{E}|c_{k}^{i}|^{p}\leq C<\infty \label{Eq:Bound_c}
\end{align}
for some constant $C$ that may depend on $p,T$, and the bound on the activation function $\sigma$. We have also used the fact that $\mathbb{E} [ |c_{0}^{i}|^p ] < K_1$ due to $E[\exp( q  | c_0 | ) ] < K_2$ for some $0 < q < \infty$ in Assumption \ref{A:Assumption1}.
Now, we turn to the bound for $\parallel w^i_k \parallel$. We start with the bound (using Young's inequality)
\begin{align}
 \parallel w^i_{k+1} \parallel &\leq  \parallel w^i_k \parallel +  \frac{C}{N} \left( |y_k| +  \frac{1}{N}\sum_{i=1}^N | c^i_k | \right) | c^i_k|    | \sigma'(w^i_k \cdot x_k ) | \parallel x_k \parallel\nonumber\\
 &\leq \parallel w^i_k \parallel +   C \left(\frac{1}{N} |y_k|^{2} + \frac{1}{N^{2}}\sum_{i=1}^N | c^i_k |^{2} + \frac{1}{N} | c^i_k|^{2}   \parallel x_k \parallel^{2}\right)\nonumber\\
 &\leq \parallel w^i_k \parallel +   C \left(\frac{1}{N} |y_k|^{2} + \frac{1}{N^{2}}\sum_{i=1}^N | c^i_k |^{2}+ \frac{1}{N} | c^i_k|^{4}+  \frac{1}{N} \parallel x_k \parallel^{4}\right),\nonumber
\end{align}
for a constant $C<\infty$ that may change from line to line. Taking now expectation, using Assumption \ref{A:Assumption1}, the a-priori bound (\ref{Eq:Bound_c}) and the fact that $k/N\leq T$ we obtain
\begin{align*}
\mathbb{E}\parallel w^i_k \parallel\leq C<\infty, 
\end{align*}
for all $i\in\mathbb{N}$ and all $k$ such that $k/N\leq T$, concluding the proof of the lemma.
\end{proof}

\section{Identification of the Limit} \label{Identification}
We consider the evolution of the empirical measure $\nu^N_k$ via test functions $f \in C^{2}_{b}(\mathbb{R}^{1+d})$. A Taylor expansion yields
\begin{eqnarray}
\la f , \nu^N_{k+1} \ra - \la f , \nu^N_k \ra &=& \frac{1}{N} \sum_{i=1}^N \bigg{(} f(c^i_{k+1}, w^i_{k+1} ) -  f(c^i_{k}, w^i_{k} )  \bigg{)} \notag \\
&=& \frac{1}{N} \sum_{i=1}^N \partial_c f(c^i_{k}, w^i_{k} ) ( c^i_{k+1} -  c^i_{k} )  + \frac{1}{N} \sum_{i=1}^N \nabla_w  f(c^i_{k}, w^i_{k} )  ( w^i_{k+1} -  w^i_{k} ) \notag \\
&+& \frac{1}{N} \sum_{i=1}^N \partial^{2}_{c} f(\bar c^i_{k},  \bar w^i_{k} ) ( c^i_{k+1} -  c^i_{k} )^2  + \frac{1}{N} \sum_{i=1}^N ( c^i_{k+1} -  c^i_{k} )\nabla_{cw}  f(\bar c^i_{k}, \bar w^i_{k} )( w^i_{k+1} -  w^i_{k} )    \notag \\
&+& \frac{1}{N} \sum_{i=1}^N ( w^i_{k+1} -  w^i_{k} )^{\top}\nabla^{2}_{w} f(\bar c^i_{k}, \bar w^i_{k} ) ( w^i_{k+1} -  w^i_{k} ),\notag
\end{eqnarray}
for points $\bar c^{i}_{k}, \bar w^{i}_{k}$ in the segments connecting $c^i_{k+1}$ with $c^i_{k}$ and  $w^i_{k+1}$ with $w^i_{k}$, respectively.
Notice now that the uniform bounds of Lemma \ref{L:AprioriBound} and the relation (\ref{Eq:SGD}) imply that as $N$ gets large
\begin{eqnarray}
\la f , \nu^N_{k+1} \ra - \la f , \nu^N_k \ra &=&  \frac{1}{N^2} \sum_{i=1}^N \partial_c f(c^i_{k}, w^i_{k} )  \alpha (y_k - g_{\theta_k}^N(x_k) )  \sigma (w^i_k \cdot x_k)   \notag \\
&+& \frac{1}{N^2} \sum_{i=1}^N   \alpha (y_k - g_{\theta_k}^N(x_k) )  c^i_k \sigma' (w^i_k \cdot x_k) \nabla_w  f(c^i_{k}, w^i_{k} )\cdot x_{k} + O_{p}\left(N^{-2}\right).\notag
\end{eqnarray}

The term $O_{p}\left(N^{-2}\right)$\footnote{ Recall that when we write $Z=O_{p}(b)$ we mean that $Z/b$ is stochastically bounded.} is a result of $f \in C^{2}_{b}$, the bounds from Lemma \ref{L:AprioriBound} as well as the moment bounds on $(x_{k},y_{k})$ from Assumption \ref{A:Assumption1}. 
We next define the drift and martingale components:
\begin{eqnarray}
D^{1,N}_k &=& \frac{1}{N} \int_{\mathcal{X}\times\mathcal{Y}}   \alpha \big{(} y -  \la c \sigma(w \cdot x),  \nu^N_k \ra \big{)} \la \sigma(w \cdot x) \partial_c f, \nu_k^N \ra \pi(dx,dy), \notag \\
D^{2,N}_k &=&  \frac{1}{N} \int_{\mathcal{X}\times\mathcal{Y}}   \alpha \big{(} y -  \la c \sigma(w \cdot x),  \nu^N_k \ra \big{)} \la c  \sigma'(w \cdot x) x \cdot \nabla_w f, \nu_k^N \ra  \pi(dx,dy), \notag \\
M^{1,N}_k &=&  \frac{1}{N} \alpha \big{(} y_k -  \la c \sigma(w \cdot x_k),  \nu^N_k \ra \big{)} \la \sigma(w \cdot x_k) \nabla_c f, \nu_k^N \ra   - D^{1,N}_k, \notag \\
M^{2,N}_k &=& \frac{1}{N} \alpha \big{(} y_k -  \la c \sigma(w \cdot x_k),  \nu^N_k \ra \big{)} \la c  \sigma'(w \cdot x_k) x \cdot \nabla_w f, \nu_k^N \ra  - D^{2,N}_k.\notag
\end{eqnarray}

Combining the different terms together, we then obtain
\begin{eqnarray*}
\la f , \nu^N_{k+1} \ra - \la f , \nu^N_k \ra &=&  D^{1,N}_k + D^{2,N}_k + M^{1,N}_k + M^{2,N}_k + O_{p}\left(N^{-2}\right).
\end{eqnarray*}

Next, we define the scaled versions of $D^{1,N}, D^{2,N}, M^{1,N}$ and $M^{2,N}$:
\textcolor{black}{\begin{eqnarray}
D^{1,N}(t) &=&  \sum_{k=0}^{ \floor*{N t}-1 } D^{1,N}_k, \qquad D^{2,N}(t) =  \sum_{k=0}^{ \floor*{N t} -1} D^{2,N}_k , \notag \\
M^{1,N}(t) &=& \sum_{k=0}^{ \floor*{N t}-1 } M^{1,N}_k, \qquad M^{2,N}(t) = \sum_{k=0}^{ \floor*{N t}-1 } M^{2,N}_k. \notag
\end{eqnarray}}

The scaled empirical measure satisfies, as $N$ grows,
\begin{eqnarray}
\la f, \mu^N_{t} \ra - \la f, \mu^N_{0} \ra &=& \int_0^t  \bigg{(} \int_{\mathcal{X}\times\mathcal{Y}}  \alpha \big{(} y -  \la c \sigma(w \cdot x),  \mu^N_{s} \ra \big{)} \la \sigma(w \cdot x) \nabla_c f, \mu^N_{s} \ra \pi(dx,dy)  \bigg{)} ds\notag \\
&+& \int_0^t \bigg{(} \int_{\mathcal{X}\times\mathcal{Y}}   \alpha \big{(} y -  \la c \sigma(w \cdot x),  \mu^N_{s} \ra \big{)} \la c  \sigma'(w \cdot x) x \cdot \nabla_w f, \mu^N_{s} \ra \pi(dx, dy) \bigg{)}ds \notag \\
&+& M^{1,N}(t) + M^{2,N}(t) + O_{p}\left(N^{-1}\right).\notag
\end{eqnarray}

In fact as we show below  $M^{1,N}(t)$ and $M^{2,N}(t)$ converge to $0$ in $L^2$ as $N \rightarrow \infty$.
\begin{lemma}\label{L:MartingaleTermsZero}
We have that
\begin{eqnarray}
\lim_{N \rightarrow \infty} \mathbb{E} \bigg{[} \bigg{(} M^{1,N}(t) \bigg{)}^2 \bigg{]} &=& 0, \notag \\
\lim_{N \rightarrow \infty} \mathbb{E} \bigg{[} \bigg{(} M^{2,N}(t) \bigg{)}^2 \bigg{]} &=& 0.\notag
\end{eqnarray}
\end{lemma}
\begin{proof}
First, notice that
\textcolor{black}{\begin{eqnarray}
&\phantom{.}& \mathbb{E} \bigg{[} \bigg{(} \sum_{k=0}^{ \floor*{N t}-1 }  \frac{1}{N} \alpha \big{(} y_k -  \la c \sigma(w \cdot x_k),  \nu^N_k \ra \big{)} \la \sigma(w \cdot x_k) \partial_c f, \nu_k^N \ra   - D^{1,N}_k \bigg{)}^2 \bigg{]} \notag \\
&=&     \sum_{j, k=0}^{ \floor*{N t}-1 } \mathbb{E} \bigg{[} \bigg{(} \frac{1}{N} \alpha \big{(} y_k -  \la c \sigma(w \cdot x_k),  \nu^N_k \ra \big{)} \la \sigma(w \cdot x_k) \partial_c f, \nu_k^N \ra   - D^{1,N}_k \bigg{)} \notag \\
&\times&  \bigg{(} \frac{1}{N} \alpha \big{(} y_j -  \la c \sigma(w \cdot x_j),  \nu^N_j \ra \big{)} \la \sigma(w \cdot x_j) \partial_c f, \nu_j^N \ra   - D^{1,N}_j \bigg{)} \bigg{]}
\label{DoubleSummationM}
\end{eqnarray}}

Let $\mathcal{F}_k^N$ be the $\sigma-$algebra generated by $(c^{i}_{0},w^{i}_{0})_{i=1}^{N}$ and $(x_{j}, y_{j})_{j=0}^{k-1}$. If $j > k$, then
\begin{eqnarray}
&\phantom{.}& \mathbb{E} \bigg{[} \bigg{(} \frac{1}{N} \alpha \big{(} y_k -  \la c \sigma(w \cdot x_k),  \nu^N_k \ra \big{)} \la \sigma(w \cdot x_k) \partial_c f, \nu_k^N \ra   - D^{1,N}_k \bigg{)} \notag \\
&\times&  \bigg{(} \frac{1}{N} \alpha \big{(} y_j -  \la c \sigma(w \cdot x_j),  \nu^N_j \ra \big{)} \la \sigma(w \cdot x_j) \partial_c f, \nu_j^N \ra   - D^{1,N}_j \bigg{)} \bigg{]}  \notag \\
&=&  \mathbb{E} \bigg{[} \bigg{(} \frac{1}{N} \alpha \big{(} y_k -  \la c \sigma(w \cdot x_k),  \nu^N_k \ra \big{)} \la \sigma(w \cdot x_k) \partial_c f, \nu_k^N \ra   - D^{1,N}_k \bigg{)} \notag \\
&\times&  \mathbb{E} \bigg{[} \bigg{(} \frac{1}{N} \alpha \big{(} y_j -  \la c \sigma(w \cdot x_j),  \nu^N_j \ra \big{)} \la \sigma(w \cdot x_j) \partial_c f, \nu_j^N \ra   - D^{1,N}_j \bigg{)}  \bigg{|} \mathcal{F}_{j-1}^N \bigg{]}  \bigg{]} \notag \\
&=&  \mathbb{E} \bigg{[} \bigg{(} \frac{1}{N} \alpha \big{(} y_k -  \la c \sigma(w \cdot x_k),  \nu^N_k \ra \big{)} \la \sigma(w \cdot x_k) \partial_c f, \nu_k^N \ra   - D^{1,N}_k \bigg{)} \times 0 \bigg{]} \notag \\
&=& 0.\notag
\end{eqnarray}
Therefore, (\ref{DoubleSummationM}) reduces to

\begin{eqnarray}
&\phantom{.}& \mathbb{E} \bigg{[} \bigg{(} \sum_{k=0}^{ \floor*{N t}-1 }  \frac{1}{N} \alpha \big{(} y_k -  \la c \sigma(w \cdot x_k),  \nu^N_k \ra \big{)} \la \sigma(w \cdot x_k) \partial_c f, \nu_k^N \ra   - D^{1,N}_k \bigg{)}^2 \bigg{]} \notag \\
&=&  \sum_{k=0}^{ \floor*{N t}-1 }  \mathbb{E} \bigg{[} \bigg{(} \frac{1}{N} \alpha \big{(} y_k -  \la c \sigma(w \cdot x_k),  \nu^N_k \ra \big{)} \la \sigma(w \cdot x_k) \partial_c f, \nu_k^N \ra   - D^{1,N}_k \bigg{)}^2 \bigg{]}.
  \label{SingleSummationM}
\end{eqnarray}

Using (\ref{SingleSummationM}), we have that
\begin{eqnarray}
\mathbb{E} \bigg{[} \bigg{(} M^{1,N}(t) \bigg{)}^2 \bigg{]} &=& \mathbb{E} \bigg{[} \bigg{(} \sum_{k=0}^{ \floor*{N t} -1}  \frac{1}{N} \alpha \big{(} y_k -  \la c \sigma(w \cdot x_k),  \nu^N_k \ra \big{)} \la \sigma(w \cdot x_k) \partial_c f, \nu_k^N \ra   - D^{1,N}_k \bigg{)}^2 \bigg{]} \notag \\
&=&   \sum_{k=0}^{ \floor*{N t} -1} \mathbb{E} \bigg{[} \bigg{(} \frac{1}{N} \alpha \big{(} y_k -  \la c \sigma(w \cdot x_k),  \nu^N_k \ra \big{)} \la \sigma(w \cdot x_k) \partial_c f, \nu_k^N \ra   - D^{1,N}_k \bigg{)}^2 \bigg{]} \notag \\
&\leq& \frac{2}{N^2} \sum_{k=0}^{ \floor*{N t} -1}  \mathbb{E} \bigg{[} \bigg{(}  \alpha \big{(} y_k -  \la c \sigma(w \cdot x_k),  \nu^N_k \ra \big{)} \la \sigma(w \cdot x_k) \partial_c f, \nu_k^N \ra \bigg{)}^2 \bigg{]}  \notag \\
&+& \frac{2}{N^2} \sum_{k=0}^{ \floor*{N t}-1 }  \mathbb{E} \bigg{[} \bigg{(}  \int_{\mathcal{X}\times\mathcal{Y}}   \alpha \big{(} y -  \la c \sigma(w \cdot x),  \nu^N_k \ra \big{)} \la \sigma(w \cdot x) \partial_c f, \nu_k^N \ra \pi(dx,dy) \bigg{)}^2 \bigg{]}  \notag \\
&\leq& \frac{C}{N^2}  \floor*{N t}. \notag
\end{eqnarray}

The final inequality comes from the bounds proven in Section \ref{RelativeCompactness} and Assumption \ref{A:Assumption1}. A similar bound can be also established for $\mathbb{E} \bigg{[} \bigg{(} M^{2,N}(t) \bigg{)}^2 \bigg{]}$. The result directly follows.
\end{proof}

Let $\pi^N$ be the probability measure of a convergent subsequence of $\left(\mu^N\right)_{0\leq t\leq T}$. Each $\pi^N$ takes values in the set of probability measures $\mathcal{M} \big{(} D_E([0,T]) \big{)}$. Relative compactness, proven in Section \ref{RelativeCompactness}, implies that there is a subsequence $\pi^{N_k}$ which weakly converges. We must prove that any limit point $\pi$ of a convergent subsequence $\pi^{N_k}$ will satisfy the evolution equation (\ref{EvolutionEquationIntroduction}).
\begin{lemma}
Let $\pi^{N_k}$ be a convergent subsequence with a limit point $\pi$. Then $\pi$ is a Dirac measure concentrated on $\bar \mu \in D_E([0,T])$ and $\bar \mu$ satisfies the measure evolution equation (\ref{EvolutionEquationIntroduction}).
\end{lemma}
\begin{proof}
We define a map $F(\mu): D_{E}([0,T]) \rightarrow \mathbb{R}_{+}$ for each $t \in [0,T]$, $f \in C^{2}_{b}(\mathbb{R}^{1+d})$, $g_{1},\cdots,g_{p}\in C_{b}(\mathbb{R}^{1+d})$ and $0\leq s_{1}<\cdots< s_{p}\leq t$.
\begin{eqnarray*}
F(\mu) &=&  \bigg{|} \left(\la f, \mu_t \ra - \la f,  \mu_0 \ra - \int_0^t   \bigg{(} \int_{\mathcal{X}\times\mathcal{Y}}   \alpha \big{(} y -  \la c' \sigma(w' \cdot x),  \mu_s \ra \big{)} \la \sigma(w \cdot x) \partial_c f, \mu_s \ra   \pi(dx,dy) \bigg{)} ds\right.\notag \\
&+& \left. \int_0^t  \bigg{(} \int_{\mathcal{X}\times\mathcal{Y}}   \alpha \big{(} y -  \la c' \sigma(w' \cdot x), \mu_s \ra \big{)} \la c \sigma'(w \cdot x) x \cdot \nabla_w f, \mu_s \ra  \pi(dx,dy)\bigg{)}ds\right)\times\nonumber\\
& &\qquad \times\la g_{1},\mu_{s_{1}}\ra\times\cdots\times \la g_{p},\mu_{s_{p}}\ra\bigg{|} .
\label{EvolutionEquation2}
\end{eqnarray*}
Then, by the proof of Lemma \ref{L:MartingaleTermsZero}, we obtain for large $N$
\begin{eqnarray}
\mathbb{E}_{\pi^N} [ F(\mu) ] &=& \mathbb{E} [ F( \mu^N ) ] \notag \\
&=& \mathbb{E} \left| \left(M^{1,N}(t) + M^{2,N}(t) + O(N^{-1})\right)\prod_{i=1}^{p} \la g_{i},\mu^{N}_{s_{i}}\ra \right| \notag \\
&\leq&  \mathbb{E}[ | M^{1,N}(t) |  ]+  \mathbb{E}[ | M^{2,N}(t)  |  ]+ O(N^{-1})   \notag \\
&\leq& \mathbb{E}[ ( M^{1,N}(t) )^2 ]^{1/2}+  \mathbb{E}[ ( M^{2,N}(t)  )^2 ]^{1/2} + O(N^{-1})  \notag \\
&\leq& C \left(\frac{1}{\sqrt{N}}+\frac{1}{N}\right).\notag
\end{eqnarray}
Therefore,
\begin{eqnarray}
\lim_{N \rightarrow \infty} \mathbb{E}_{\pi^N} [ F(\mu) ] = 0.\notag
\end{eqnarray}
Since $F(\cdot)$ is continuous and $F( \mu^N)$ is uniformly bounded (due to the uniform boundedness results of Section \ref{RelativeCompactness}),
\begin{eqnarray}
 \mathbb{E}_{\pi} [ F(\mu) ] = 0.\notag
\end{eqnarray}
Since this holds for each $t \in [0,T]$, $f \in C^{2}_{b}(\mathbb{R}^{1+d})$ and $g_{1},\cdots,g_{p}\in C_{b}(\mathbb{R}^{1+d})$, $\bar \mu$ satisfies the evolution equation (\ref{EvolutionEquationIntroduction}).
\end{proof}

It remains to prove that the evolution equation (\ref{EvolutionEquationIntroduction}) has a unique solution. This is the content of Section \ref{Uniqueness}.

\section{Uniqueness} \label{Uniqueness}
We prove uniqueness of a solution to  the evolution equation (\ref{EvolutionEquationIntroduction}). We will set up a Picard type of iteration and prove that it has a unique fixed point through a contraction mapping. We start by noticing that we can write
\begin{align}
\la f, \bar \mu_t \ra  &= \la f, \bar \mu_0 \ra + \int_0^t   \bigg{(} \int_{\mathcal{X}\times\mathcal{Y}}   \alpha \big{(} y -  \la c' \sigma(w'\cdot x),  \bar \mu_s \ra \big{)} \la \sigma(w \cdot x) \partial_c f, \bar \mu_s \ra   \pi(dx,dy)\bigg{)} ds \notag \\
&+ \int_0^t \bigg{(} \int_{\mathcal{X}\times\mathcal{Y}}  \alpha \big{(} y -  \la c' \sigma(w' \cdot x), \bar  \mu_s \ra \big{)} \la c \sigma'(w \cdot x) x \nabla_w f, \bar \mu_s \ra \pi(dx, dy) \bigg{)}   ds. \nonumber\\
&= \la f, \bar \mu_0 \ra+\int_{0}^{t} \la G(z, Q(\bar \mu_{s},\cdot))\cdot\nabla f,\bar{\mu}_{s}\ra ds,
\label{EvolutionEquation1}
\end{align}
where for $z=(c,w_{1},\cdots,w_{d})\in\mathbb{R}^{1+d}$, $Q(\bar \mu,x)=\la c \sigma(w \cdot x), \bar \mu \ra$ we have
\[
G(z,Q(\bar \mu,\cdot))=(G_{1}(z,Q(\bar \mu,\cdot)), G_{2}(z,Q(\bar \mu,\cdot)))\in \mathbb{R}^{1+d}
\]
 with
\begin{align*}
G_{1}(z,Q(\bar \mu,\cdot))&= \int_{\mathcal{X}\times\mathcal{Y}}\alpha ( y - Q(\bar \mu,x) )  \sigma(w \cdot x) \pi(dx,dy)\in\mathbb{R}\nonumber\\
G_{2}(z,Q(\bar \mu,\cdot))&= \int_{\mathcal{X}\times\mathcal{Y}}\alpha ( y - Q(\bar \mu,x) )  c \sigma'(w \cdot x) x \pi(dx,dy)\in\mathbb{R}^{d}.\notag
\end{align*}

We remark here that a solution to (\ref{EvolutionEquation1}), $\bar{\mu}_{\cdot}$, is associated to the nonlinear random process $Z_{t}$ (see for example \cite{Kolokoltsov}) satisfying  the random ordinary differential equation (ODE)
\begin{align}
Z_{t}&=Z_{0}+\int_{0}^{t} G(Z_{s},Q(\bar \mu_{s},\cdot))ds\nonumber\\
Z_{0}&\sim  \bar \mu(0,c,w)\nonumber\\
\bar{\mu}_{t}&=\text{Law}(Z_{t})\label{Eq:RandomODE}
\end{align}
This ODE is random due to the random initial  data.

Let us now define the following mappings. Let $F: D([0,T];\mathbb{R}) \mapsto D([0,T];M(\mathbb{R}^{1+d}))$ be such that for a path $(R_{t})_{t\in[0,T]}\in D([0,T];\mathbb{R})$, we have that $F(R_{\cdot})=\text{Law}(Y_{\cdot})$ where $Y_{\cdot}$ is given by
\begin{align}
Y_{t}&=Y_{0}+\int_{0}^{t} G(Y_{s},R_{s})ds\nonumber\\
Y_{0}&\sim \bar \mu(0,c,w).\notag
\end{align}

 Now, let us also define the map $L:D([0,T];M(\mathbb{R}^{1+d}))\mapsto D([0,T];\mathbb{R})$ taking a measure valued process $\mu_{t}$ and mapping it to $Q(\mu_{t},x)=L(\mu)$ where
\[
Q(\mu_{t},x)=\la  c \sigma(w \cdot x),  \mu_{t} \ra.
\]

Then, we consider the mapping $H:D([0,T];M(\mathbb{R}^{1+d}))\mapsto D([0,T];M(\mathbb{R}^{1+d}))$ defined via the composition of the mappings $F$ and $L$, we set $H=F\circ L$. Sometimes, in order to emphasize the dependence on $T$, we may write $H_{T}$ for $H$.

It is clear that if $(\mu_{t})_{t\in[0,T]}$ is a fixed point of $H$, then $\text{Law}(Z_{t})=H_{t}(\mu_{\cdot})$ is a solution to (\ref{EvolutionEquation1}). Conversely, if  $(Z_{t})_{t\in[0,T]}$ is a solution to (\ref{Eq:RandomODE}) then its law will be a fixed point of $H$, implying that $\text{Law}(Z_{t})=H_{t}(\mu)$. In addition, if $\mu$ is a weak measure valued solution to (\ref{EvolutionEquation1}), then it must be a fixed point of $H$ and thus satisfy  (\ref{Eq:RandomODE}),   proving our result.

Now, we need to show that $H$ is a contraction mapping for $t\in[0,T]$. The first step is to show that in studying the fixed point of $H$, we can in fact consider $H:C([0,T];M(\mathbb{R}^{1+d}))\mapsto C([0,T];M(\mathbb{R}^{1+d}))$. This will allow us to work in $C([0,T];M(\mathbb{R}^{1+d}))$ instead of working in the larger space $D([0,T];M(\mathbb{R}^{1+d}))$ streamlining some elements of the proof.

For this reason we first derive some a-priori bounds and study regularity for  $Z_{t}$ satisfying the random ODE given by (\ref{Eq:RandomODE}) where $\bar\mu_{t}$ is the probability measure of the parameters at time $t$. Denoting by $\mathbb{E}$ the expectation operator taken with respect to this measure (notice that here $(x,y)$ are considered to be integration variables) we essentially consider the following system of random ODE's.
\begin{eqnarray}
 c_t &=& c_{0} +\int_{0}^{t}\alpha \int_{\mathcal{X}\times\mathcal{Y}}   (y - \mathbb{E}[ c_s \sigma (w_s \cdot x)] )  \sigma (w_s \cdot x) \pi(dx,dy) ds, \notag \\
w_t &=& w_{0} +\int_{0}^{t}\alpha \int_{\mathcal{X}\times\mathcal{Y}}  (y - \mathbb{E}[ c_s \sigma( w_s \cdot x) ] ) c_s \sigma' (w_s \cdot x) x  \pi(dx,dy) ds. \notag \\
(c_0, w_0) &\sim& \bar \mu(0,c,w).
\label{SDEmain}
\end{eqnarray}

Lemma \ref{L:regularityLimitODE} shows that there is regularity in time and it also provides us with some useful a-priori uniform bounds.
\begin{lemma}\label{L:regularityLimitODE}
Let $1\leq p\leq 4$ and $T<\infty$ be given. Then, there are constants $C_{1},C_{2},C_{3}<\infty$, depending on $p$, such that
\[
 \sup_{t\in[0,T]}|c_{t}|^{p} \leq C_{1}\left(|c_{0}|^{p}+1+T^{p}C_{2}\right), \quad \mathbb{E}\sup_{t\in[0,T]}\parallel w_{t}\parallel^{p} \leq C_{1}\left(\mathbb{E}|c_{0}|^{p} + \mathbb{E}\parallel w_{0}\parallel^{p}+T^{p}C_{2}\right)
\]
and for every $0\leq s\leq t\leq T$ we have that
\[
|c_{t}-c_{s}|^{p} \leq C |t-s|^{p},\quad  \mathbb{E}\parallel w_{t}-w_{s}\parallel^{p}\leq C\left( \mathbb{E}|c_0|^{p}+1 \right) |t-s|^{p}.
\]
\end{lemma}
\begin{proof}
Let's examine $c_t$ first and establish a bound on its growth. The constant $C$ may change from line to line and it may also depend upon the final time $T$ and on $p$.
\begin{eqnarray}
 c_t &=&  c_0+ \int_0^t \alpha \int_{\mathcal{X}\times\mathcal{Y}}     (y - \mathbb{E}[ c_s \sigma (w_s \cdot x)] )  \sigma (w_s \cdot x) \pi(dx,dy) ds. \notag \\
 c_t  \sigma(w_t \cdot x) &=&  \sigma(w_t \cdot x) c_0+ \sigma(w_t \cdot x) \int_0^t \alpha \int_{\mathcal{X}\times\mathcal{Y}}     (y - \mathbb{E}[ c_s \sigma (w_s \cdot x)] )  \sigma (w_s \cdot x)\pi(dx,dy)  ds. \notag \\
  | c_t  \sigma(w_t \cdot x) | &\leq& C | c_0 |+  C \int_0^t \int_{\mathcal{X}\times\mathcal{Y}}     | (y - \mathbb{E}[ c_s \sigma (w_s \cdot x)] )  |  \pi(dx,dy) ds. \label{Eq:c_randomODEbound}
\end{eqnarray}
We have used the fact that $\sigma(\cdot)$ is bounded. Now, we will use the facts that $\mathbb{E}[ | c_0 |^{p} ] < C$ and $x,y$ have finite moments when integrated against $\pi(dx,dy)$ via Assumption \ref{A:Assumption1}. For a potentially different constant $C<\infty$
\begin{eqnarray}
\mathbb{E} [ | c_t  \sigma(w_t \cdot x) |^{p}  ]&\leq& C + C \int_0^t \int_{\mathcal{X}\times\mathcal{Y}}  \mathbb{E}[ | c_s \sigma(w_s \cdot x ) |^{p} ]  \pi(dx,dy)   ds. \notag \\
   \int_{\mathcal{X}\times\mathcal{Y}}\mathbb{E} [ | c_t  \sigma(w_t \cdot x) |^{p} ] \pi(dx,dy)&\leq& C + C \int_0^t  \int_{\mathcal{X}\times\mathcal{Y}}  \mathbb{E}[ | c_s \sigma(w_s \cdot x) |^{p} ]  \pi(dx,dy)   ds. \notag
\end{eqnarray}
Therefore, by Gronwall's inequality, there exists a constant $C_{2}<\infty$ such that
\begin{eqnarray*}
 \int_{\mathcal{X}\times\mathcal{Y}}  \mathbb{E}[ | c_s \sigma(w_s \cdot x) |^{p} ]  \pi(dx,dy)  \leq C_{2},
 \label{BoundMomentC000011}
\end{eqnarray*}
for $0 \leq s \leq T$. Therefore, returning to (\ref{Eq:c_randomODEbound}) and recalling Assumption \ref{A:Assumption1} we get that uniformly in $t\in[0,T]$, there exist constants $C_{1},C_{2}<\infty$ such that
\begin{eqnarray*}
 \sup_{t\in[0,T]}|c_{t}|^{p} \leq C_{1}\left(|c_{0}|^{p}+1+T^{p} C_{2}\right).
\end{eqnarray*}

Let us now obtain the claimed bound on $\mathbb{E}\sup_{t\in[0,T]}\parallel w_{t}\parallel^{p}$. We obtain from (\ref{SDEmain}) and Assumption \ref{A:Assumption1} that
\begin{align}
\parallel w_t\parallel^{p} &\leq \parallel w_{0}\parallel^{p} +t^{p-1}\int_{0}^{t}\alpha^{p} \left(\int_{\mathcal{X}\times\mathcal{Y}}  (y - \mathbb{E}[ c_s \sigma( w_s \cdot x) ] ) c_s \sigma' (w_s \cdot x) x  \pi(dx,dy)\right)^{p} ds\nonumber\\
&\leq \parallel w_{0}\parallel^{p} + C t^{p-1}\int_{0}^{t} \left(\int_{\mathcal{X}\times\mathcal{Y}}  |y - \mathbb{E}[ c_s \sigma( w_s \cdot x) ] |^{2}\pi(dx,dy)\right)^{p/2} \left(\int_{\mathcal{X}\times\mathcal{Y}} |c_s|^{2} \parallel x\parallel^{2}  \pi(dx,dy)\right)^{p/2} ds\nonumber\\
&\leq \parallel w_{0}\parallel^{p} + C t^{p-1}\int_{0}^{t} |c_s|^{p} \left(\int_{\mathcal{X}\times\mathcal{Y}}  (|y|^{2} + \mathbb{E}[| c_s \sigma( w_s \cdot x)  |^{2}]\pi(dx,dy)\right)^{p/2} \left(\int_{\mathcal{X}\times\mathcal{Y}}  \parallel x\parallel^{2}  \pi(dx,dy)\right)^{p/2} ds\nonumber
\end{align}
 and the claimed bound follows by taking supremum over all $t\in[0,T]$, expectation and using the previously derived uniform bound for $c_{t}$.
%
%
Let us now prove the second statement of the lemma.  Similarly to the calculations above and using the uniform moment bounds on $c_{t}$ and $w_{t}$ together with Assumption \ref{A:Assumption1}, we have
\begin{align}
 |c_t-c_{s}|^{p}  &\leq C |t-s|^{p-1} \int_s^t  \int_{\mathcal{X}\times\mathcal{Y}}     |y - \mathbb{E}[ c_u \sigma (w_u \cdot x)] |^{p}   \pi(dx,dy) du \notag \\
& \leq C_{3} |t-s|^{p},\notag
\end{align}
 for some unimportant constant $C_{3}<\infty$. Using an analysis almost identical to the one that led to the bound for $\mathbb{E}\sup_{t\in[0,T]}\parallel w_{t}\parallel^{p}$ gives us the desired bound for $\mathbb{E}\parallel w_{t}-w_{s}\parallel^{p}$.
This concludes the proof of the lemma.
\end{proof}
As a consequence of the regularity result in Lemma \ref{L:regularityLimitODE}, (\ref{SDEmain}) is a continuous process. Therefore, we can prove a contraction in $C([0,T];M(\mathbb{R}^{1+d}))$ (instead of studying the process in the larger space $D([0,T];M(\mathbb{R}^{1+d}))$).

Let us define for notational convenience $C_T= C([0,T], \mathbb{R}^{1+d})$ and let $M_T$ be the set of probability measures on $C_T$. Consider an element $m\in M_{T}$. Motivated by the discussion before Lemma \ref{L:regularityLimitODE} let us set $\text{Law}(Y)=H(m_{\cdot})$, where, slightly abusing notation, $Y=(c,w)$ with
\begin{eqnarray}
c_t &=&c_0 +    \int_0^t \int_{\mathcal{X}\times\mathcal{Y}}  \alpha  \bigg{(} y - \la G_{s,x},m \ra \bigg{)}  \sigma ( w_s \cdot x) \pi(dx,dy) ds, \notag \\
w_t &=& w_0 +  \int_0^t   \int_{\mathcal{X}\times\mathcal{Y}} \alpha \bigg{(} y - \la G_{s,x},m \ra \bigg{)} c_s \sigma' (  w_s\cdot x) x \pi(dx,dy) ds, \notag \\
G_{s,x} &=& c'_s \sigma (w'_s \cdot x), \notag \\
( c_0, w_0) &\sim& \bar \mu(0,c,w).
\label{SDEY}
\end{eqnarray}

For $m, m' \in M_T$ and $p\geq 1$ define the metric
\begin{eqnarray*}
D_{T,p}(m, m') = \inf  \bigg{\{} \bigg{(} \int_{C_T \times C_T} \sup_{s \leq T} \norm{ x_s - y_s }_{p}^{p}\wedge 1 d \nu (x, y)  \bigg{)}^{1/p}, \nu \in P(m, m') \bigg{\}},
\end{eqnarray*}
where $P(m, m')$ is the set of probability measures on $C_T \times C_T$ such that the marginal distributions are $m$ and $m'$, respectively.

Now we show  existence and uniqueness of a fixed point $\text{Law}( c_t,  w_t)$ for the mapping $H$, as defined via (\ref{SDEY}). If a solution to (\ref{Eq:RandomODE}) exists, then it must be a fixed point of $H$ (defined via equation (\ref{SDEY})). This is an immediate consequence of Lemma \ref{L:regularityLimitODE}. Therefore, if $H$ has a unique solution, there can be at most one solution to (\ref{Eq:RandomODE}). If (\ref{Eq:RandomODE}) has at most one solution, (\ref{EvolutionEquation1}) has at most one solution. Therefore, if $H$ has a unique fixed point, this proves uniqueness for (\ref{EvolutionEquation1}).

Due to Lemma \ref{L:regularityLimitODE} we need only to consider the space of measures $M_T$ that have bounded moments up to order $p=4$. By known results, see for example \cite{Bolley}, the space of measures with $p=4$ finite moments endowed with the $D_{T,4}$ metric is complete and separable. Due to closedness, the space of measures with bounded moments of order four is  a complete and separable metric space when endowed with the Wasserstein metric $D_{T,4}$. Therefore, in the arguments below we work with the space of measures that have bounded the first four moments and we consider the metric  $D_{T,4}$. We will show that there is a unique fixed point by proving a contraction.

Lemma \ref{L:MapToItself} shows that for a large enough bound, $H(m)$ maps from a subspace of bounded moments to the same subspace of bounded moments.
\begin{lemma}\label{L:MapToItself}
Consider $(c_{t}, w_{t})$ solving (\ref{SDEY}). There is a $K_{0}$ such that for any $K>K_{0}$, we can find $T_{1}=T_{1}(K)<T$ such that $\la \sup_{0 \leq t \leq T_1} ( |c_t|^4 + \norm{w_t}^4) , m \ra  < K$ implies $\mathbb{E}[ \sup_{0 \leq t \leq T_1} (|c_t|^4 + \norm{w_t}^4) ] < K$.
\end{lemma}
\begin{proof}
\textcolor{black}{Assume that the measure $m$ is such that $\la \sup_{0 \leq t \leq T} ( |c_t|^4 + \norm{w_t}^4) , m \ra  < K$ (where $K<\infty$ will be chosen below). Using the same steps as in Lemma \ref{L:regularityLimitODE}, we can show that for some $T_{1}<T$ (to be chosen later),}
\begin{align}
\mathbb{E} [ \sup_{0 \leq t \leq T_1}  | c_t |^4 ] &\leq C_{1} \big{(} 1 + T_1^{4} K \big{)}, \notag \\
\mathbb{E} [ \sup_{0 \leq t \leq T_1}  \norm{ w_t }^4 ] &\leq C_{1} \big{(} 1 + T_1^{4} K \mathbb{E} [ \sup_{0 \leq t \leq T_{1}}  | c_t |^4 ]  \big{)}.\notag
\end{align}

Let $K_{0}>2 C_{1}$ and let $K>K_{0}$. Then, if $T_1 \leq \left(\frac{K - 2 C_{1}}{2 C_{1} K}\right)^{1/4}$, $\mathbb{E} [ \sup_{0 \leq t \leq T_1}  | c_t |^4 ]  \leq \frac{K}{2}$.  If, in addition,  we have $T_1 \leq \left(\frac{K - 2 C_{1}}{2 C_{1} K^2}\right)^{1/4}$, then we get $\mathbb{E} [ \sup_{0 \leq t \leq T_1}  \norm{ w_t }^4 ]  \leq \frac{K_1}{2}$. Therefore, if
\begin{eqnarray}
T_1 \leq \min\left\{ \left(\frac{K - 2 C_{1}}{2 C_{1} K}\right)^{1/4}, \left(\frac{K - 2 C_{1}}{2 C_{1} K^2}\right)^{1/4} \right\},\notag
\end{eqnarray}
we have that  $\mathbb{E}[ \sup_{0 \leq t \leq T_1} (|c_t|^4 + \norm{w_t}^4) ] < K$, concluding the proof of the lemma.

\end{proof}

We can now prove a contraction and then apply the Banach fixed-point theorem to prove that there is a unique fixed point.

For two elements $m^{1},m^{2}\in M_{T}$, let us set $\text{Law}(Y^{i}_{\cdot})=\text{Law}((c^{i}_{\cdot},w^{i}_{\cdot}))=H(m^{i}_{\cdot})$ for $t\in[0,T]$ with $i=1,2$. So, let $(c_t^1, w_t^1)$ satisfying (\ref{SDEY}) with $m = m^1$ and $(c_t^2, w_t^2)$ satisfying (\ref{SDEY}) with $m = m^2$. The processes $(c_t^1, w_t^1)$ and $(c_t^2, w_t^2)$ have the same initial conditions. That is,
\begin{eqnarray*}
(c_0^1, w_0^1) &=& (c_0^2, w_0^2) = (c_0, w_0), \notag \\
(c_0, w_0) &\sim& \bar \mu(0,dc,dw).
\end{eqnarray*}

We now prove a contraction for the mapping $H$ for some $0 < T_0 < T$. By definition, $(c_t^1, w_t^1)$ and $(c_t^2, w_t^2)$ have marginal distributions $H(m^1)$ and $H(m^2)$, respectively, on the time interval $[0, T_0]$. Once this is proven, we can extend this to the entire interval $[0,T]$ since $T_{0}$ is not affected by the input measures $m^{1},m^{2}$ or by which subinterval of $[0,T]$ we are considering. We have the following lemma.

\begin{lemma} \label{Contraction1}
Let $m^{1},m^{2}\in M_{T}$ and $T<\infty$. Then, there exist constants $C_1, C_2 <\infty$ that may depend on $T$ such that
\[
D_{t,4}^4(H(m^{1}), H(m^{2})) \leq C_2  t^{4} D_{t,4}^4(m^1, m^2) \bigg{(}\mathbb{E}  e^{    C_1 t |c_0| } \bigg{)} ^{1/2},
\]
for any $0 < t < T$. In addition, if $\bar \mu(0,dc,dw)$ has compact support, there exists a constant $C < \infty$ that may depend on $T$ such that
\[
D_{t,4}^4(H(m^{1}), H(m^{2})) \leq C  t^{4} D_{t,4}^4(m^1, m^2).
\]
\end{lemma}


\begin{proof}
Using the formula (\ref{SDEY}) we obtain
\begin{align}
 c_t^1 - c_t^2  &=  \int_0^t  \int_{\mathcal{X}\times\mathcal{Y}}   \alpha \bigg{(} y - \la G_{s,x},m^{1}\ra \bigg{)}  \sigma ( w_s^1 \cdot x) \pi(dx,dy) ds \nonumber\\
 &\quad -\int_0^t \int_{\mathcal{X}\times\mathcal{Y}}   \alpha \bigg{(} y - \la G_{s,x},m^{2}\ra \bigg{)}  \sigma ( w_s^2 \cdot x) \pi(dx,dy) ds, \notag \\
  &=  \int_0^t  \int_{\mathcal{X}\times\mathcal{Y}}\alpha y \bigg{(}  \sigma (w_s^1 \cdot x) -  \sigma ( w_s^2 \cdot x) \bigg{)} \pi(dx,dy) ds  \nonumber\\
  &\quad +    \int_0^t  \int_{\mathcal{X}\times\mathcal{Y}}\alpha  \la G_{s,x},m^{2}\ra  \left(\sigma (w_s^2 \cdot x)-\sigma (w_s^1 \cdot x)\right)\pi(dx,dy) ds   \notag \\
 & \quad +  \int_0^t  \int_{\mathcal{X}\times\mathcal{Y}} \alpha \la G_{s,x},m^{2}-m^{1}\ra  \sigma ( w_s^1 \cdot x) \pi(dx,dy) ds \notag
\end{align}

First, let's address the mean-field term. Recall that $\sigma'(\cdot)$ is bounded and that $\pi(dx,dy)$ has bounded marginal moments via Assumption \ref{A:Assumption1}. Therefore, for $0 \leq t \leq T$,
\begin{eqnarray*}
&\phantom{.}& \bigg{|} \int_0^t  \int_{\mathcal{X}\times\mathcal{Y}}   \la c'_s \sigma (w'_s x), m^{2}\ra \bigg{(}  \sigma (w_s^2 \cdot x) -  \sigma ( w_s^1\cdot x) \bigg{)}\pi(dx,dy)  ds \bigg{|} \notag \\
&\leq& C  \int_0^t  \int_{\mathcal{X}\times\mathcal{Y}} \big{|}  \la c'_s \sigma (w'_s x), m^{2}\ra \big{|}  \parallel x \parallel  \pi(dx,dy) \parallel w_s^2 -  w_s^1 \parallel  ds   \notag \\
&\leq& C   \int_0^t     \la | c'_s | , m^{2}\ra   \parallel w_s^2 -  w_s^1 \parallel   \left(\int_{\mathcal{X}\times\mathcal{Y}}   \parallel x \parallel \pi(dx,dy)\right) ds  
\leq C \int_0^t \parallel w_s^2 -  w_s^1 \parallel ds.
\end{eqnarray*}

We next bound the term $\bigg{|} \int_0^t   \int_{\mathcal{X}\times\mathcal{Y}}  \la c'_s \sigma (w'_s x) ,m^{2}-m^{1} \ra \sigma ( w_s^1\cdot x) \pi(dx,dy) ds \bigg{|}.$
We have that
\begin{eqnarray*}
| c^2 \sigma (w^2 \cdot x)   -  c^1 \sigma (w^1\cdot  x)  |   &=&  \big{|} (c^2 - c^1)  \sigma (w^2 \cdot x)  + c^1 \big{(} \sigma( w^2 \cdot x) - \sigma(w^1\cdot x ) \big{)}  \big{|} \notag \\
 &\leq& C_1 |c^2 - c^1|   + C_2 \parallel x \parallel  |c^1|  \parallel w^2 - w^1  \parallel.
\end{eqnarray*}

Let the random variables $(c_s^{2,'}, w_s^{2,'})$ have marginal distribution $m^2$ and $(c_s^{1,'}, w_s^{1,'})$ have marginal distribution $m^1$. Then, for $0 \leq s  \leq T$,
\begin{eqnarray}
&\phantom{.}& \bigg{|} \int_0^t  \int_{\mathcal{X}\times\mathcal{Y}}  \la c'_s \sigma (w'_s x) , m^{2}-m^{1}\ra   \sigma ( w_s^1\cdot x) \pi(dx,dy) ds \bigg{|} 
 \notag \\
&\leq& C  \int_0^t  \int_{\mathcal{X}\times\mathcal{Y}} \mathbb{E} \bigg{[}  \bigg{|} c_s^{2,'}  \sigma (w_s^{2,'} \cdot x )  - c_s^{1,'}  \sigma (w_s^{1,'} \cdot x )  \bigg{|} \bigg{]}   \pi(dx,dy) ds \notag \\
&\leq& C  \int_0^t  \int_{\mathcal{X}\times\mathcal{Y}} \mathbb{E} \bigg{[}  |c_s^{2,'} - c_s^{1,'} |   + \parallel x \parallel  |c^{1,'}_s|  \parallel w^{2,'}_s - w^{1,'}_s  \parallel \bigg{]}   \pi(dx,dy) ds, \notag \\
&\leq& C  \int_0^t  \mathbb{E} \bigg{[}  |c_s^{2,'} - c_s^{1,'} |   +  |c^{1,'}_s|  \parallel w^{2,'}_s - w^{1,'}_s  \parallel \bigg{]}  ds, \nonumber\\
 &\leq& C  \int_0^t  \mathbb{E} \bigg{[} (1 + |c^{1,'}_s| ) \big{(}  |c_s^{2,'} - c_s^{1,'} |   +    \parallel w^{2,'}_s - w^{1,'}_s  \parallel \big{)} \bigg{]}  ds,
\label{BoundUniqueProof0000AABB}
\end{eqnarray}
where again Assumption \ref{A:Assumption1} was used for the moment bounds of $\pi(dx,dy)$. The inequality holds for any joint distribution $\gamma(m^1, m^2)$ of $m^1$ and $m^2$. Then, the auxiliary calculations provided in  Appendix \ref{AdditionalTechnicalDetailsAppendixB} show that (\ref{BoundUniqueProof0000AABB}) can be bounded in terms of $D_{s,4}(m^1, m^2)$:
\begin{eqnarray}
 \bigg{|} \int_0^t  \int_{\mathcal{X}\times\mathcal{Y}}  \la c'_s \sigma (w'_s x) , m^{2}-m^{1}\ra   \sigma ( w_s^1\cdot x) \pi(dx,dy) ds \bigg{|} \leq C  \int_0^t D_{s,4}(m^1, m^2) ds.\label{BoundUniqueProof0000AABB2}
\end{eqnarray}

Thus, we overall get that there is a constant $C<\infty$ such that
\begin{eqnarray}
| c_t^2 - c_t^1 | \leq C  \int_0^t    \bigg{[} \parallel w_s^2 - w_s^1 \parallel +  D_{s,4}(m^1, m^2)  \bigg{]} \nonumber ds.
\end{eqnarray}

Similar calculations also give the necessary bound for $\parallel w^{1}_{t}-w^{2}_{t} \parallel$. For completeness, the details are provided in Appendix \ref{AdditionalTechnicalDetails}.
\begin{eqnarray}
\parallel w_t^2 - w_t^1 \parallel \leq C  \int_0^t  \bigg{[}  \big{|} c_s^2 - c_s^1 \big{|}   + |c_s^1|   \parallel w_s^2 - w_s^1 \parallel +  D_{s,4}(m^1, m^2) | c_s^2|  \bigg{]} ds.\label{Eq:BoundForWt}
\end{eqnarray}

Hence, for $0 \leq s \leq T$, we have the bound

\begin{align}
| c_s^2 - c_s^1  | + \parallel w_s^2 - w_s^1  \parallel   &\leq C_1 \int_0^s (1 + |c_u^1| ) \bigg{[}   | c_u^2 - c_u^1 |  +  \parallel w_u^2 -  w_u^1 \parallel  \bigg{]} du  +  C_2   \int_{0}^{s} (1 + |c_u^2| )D_{u,2}(m^1, m^2) du.  \notag 
\end{align}

Setting for notational convenience $Z_s=| c_s^2 - c_s^1  | + \parallel w_s^2 - w_s^1  \parallel$, the latter relation gives
\begin{align}
\sup_{s\leq t} Z_s   &\leq C_1 \int_0^t (1 + |c_u^1| ) \sup_{r\leq u}Z_r du  +  C_2   \int_{0}^{t} (1 + |c_u^2| )D_{u,4}(m^1, m^2) du,  \notag \end{align}
which by Gronwall lemma gives
\begin{align}
\sup_{s\leq t}Z_s & \leq  C_2 \bigg{(}   \int_{0}^{t} (1 + |c_u^2| ) D_{u,4}(m^1, m^2) du \bigg{)} \exp  \bigg{(}  C_1\int_0^t   (1 + |c_u^1| ) du \bigg{)}\nonumber\\
& \leq  C_2 D_{t,4}(m^1, m^2) \bigg{(}   \int_{0}^{t} (1 + |c_u^2| )  du \bigg{)} \exp  \bigg{(}  C_1\int_0^t   (1 + |c_u^1| ) du \bigg{)}\nonumber
\end{align}
where we used the fact that $u\mapsto D_{u,2}(m^1, m^2)$ is monotonically increasing.

Now, note that $|c_s^1|$ can be bounded in terms of $|c_0^1|$, the initial condition. In particular, using  the boundedness of $\sigma(\cdot)$, the bound on $|  \la G_{s,x},m^1 \ra |$, the moments bounds for the distribution $\pi(dx,dy)$, and the fact that $0 \leq t \leq T < \infty$, we get for some constant $C<\infty$ that changes from line to line
\begin{eqnarray}
|c_t^1 | &\leq& |c_0^1 | +  C   \int_0^t \int_{\mathcal{X}\times\mathcal{Y}}   \left| y - \la G_{s,x},m^1 \ra \right| \left| \sigma ( w_s^1 \cdot x)\right| \pi(dx,dy) ds  \notag \\
&\leq&  |c_0^1| +  C.\nonumber
\end{eqnarray}

This bound holds for any $0 \leq t \leq T$.  Then,
\begin{align}
\sup_{s\leq t}Z_s
& \leq  C_2 D_{t,4}(m^1, m^2) t (1 + |c_0^2| )   e^{  t C_1  (1 + |c_0^1| ) }\nonumber
\end{align}

Raising this to the power four gives for some constants $C_1,C_2<\infty$ different than before
\begin{align}
\sup_{s\leq t}Z_s^{4}
& \leq  C_2 D^{4}_{t,4}(m^1, m^2) t^{4} (1 + |c_0^2|^{4} )   e^{  t C_1  (1 + |c_0^1| ) }\nonumber
\end{align}

Now notice that
\begin{align}
\sup_{ s \leq t} \big{[} | c_s^1 - c_s^2  |^4 + \parallel w_s^1 - w_s^2  \parallel_4^4 \big{]} \wedge 1  &\leq \sup_{ s \leq t} \big{[} | c_s^1 - c_s^2  |^4 + \parallel w_s^1 - w_s^2  \parallel_4^4 \big{]}\nonumber\\
&\leq
\sup_{ s \leq t} \big{[} | c_s^1 - c_s^2  |^{4} + \parallel w_s^1 - w_s^2  \parallel_1^4 \big{]} \nonumber\\
&\leq \sup_{ s \leq t} Z_s^{4}.\nonumber
 \end{align}

Combining the last two displays yields
\begin{align}
\sup_{s\leq t}\big{[} | c_s^1 - c_s^2  |^4 + \parallel w_s^1 - w_s^2  \parallel_4^4 \big{]} \wedge 1
& \leq  C_2 D^{4}_{t,4}(m^1, m^2) t^{4} (1 + |c_0^2|^{4} )   e^{  t C_1  (1 + |c_0^1| ) }\nonumber
\end{align}

Next, we take expectation and apply the Cauchy-Schwartz inequality on the right hand side. We obtain
\begin{align}
\mathbb{E}\sup_{s\leq t}\big{[} | c_s^1 - c_s^2  |^4 + \parallel w_s^1 - w_s^2  \parallel_4^4 \big{]} \wedge 1
& \leq  C_2 D^{4}_{t,4}(m^1, m^2) t^{4} \mathbb{E}\left[(1 + |c_0^2|^{4} )   e^{  t C_1  (1 + |c_0^1| ) }\right]\nonumber\\
& \leq  C_2 D^{4}_{t,4}(m^1, m^2) t^{4} \left(\mathbb{E}(1 + |c_0^2|^{8} )\right)^{1/2}  \left(\mathbb{E} e^{  t 2 C_1  (1 + |c_0^1| ) }\right)^{1/2}.\nonumber
\end{align}

The latter concludes the proof due to Assumption \ref{A:Assumption1}.

\end{proof}

By Assumption \ref{A:Assumption1} we have that there exists a $0 < q < \infty$ such that the moment generating function exists, i.e.
\begin{eqnarray}
 M_C(q) \vcentcolon = \mathbb{E} \bigg{[} \exp  \bigg{(}  q |c_0 | \bigg{)} \bigg{]} < C_M.\nonumber
\end{eqnarray}

Lemma \ref{Contraction1} immediately proves there is a contraction on the interval $[0, T_0]$. Indeed,
\begin{eqnarray}
D_{t,4}^{4}(H(m^{1}), H(m^{2})) &\leq& C_2  t^{4} D_{t,4}^{4}(m^1, m^2) \bigg{[}\mathbb{E}  \exp  \bigg{(}  C_1 t |c_0^1|  \bigg{)} \bigg{]}^{1/2}.\nonumber
\end{eqnarray}
Therefore,
\begin{eqnarray}
D_{t,4}(H(m^{1}), H(m^{2})) \leq C_2^{1/4} t D_{t,4}( m^{1}, m^{2}) \bigg{[}\mathbb{E}  \exp  \bigg{(}  C_1 t |c_0^1|  \bigg{)} \bigg{]}^{1/8}.\nonumber
\end{eqnarray}

Then, choose $T_0$ such that $C_2^{1/4}  C_M^{1/8} T_0  < 1$, $C_1 T_0 < q$,and $T_0 \leq T_1$, where $T_{1}$ is from Lemma \ref{L:MapToItself}.   That is, we choose
$T_0 < \min\left\{ \frac{1}{C_2^{1/4} C_M^{1/8}}, \frac{q}{C_1}, T_1 \right\}.$

If $T_0 < \min\left\{ \frac{1}{C_2^{1/4} C_M^{1/8}}, \frac{q}{C_1}, T_1 \right\}$, then $D_{T_0, 4}(H(m^{1}), H(m^{2})) \leq k  D_{T_0,4}(m^1, m^2)$ for $k < 1$ and we have proven uniqueness on the sub-interval $[0, T_0]$ (via the Banach fixed-point theorem).

In fact, this directly proves our next Lemma \ref{Contraction2} regarding uniqueness of the limit point over the entire interval $[0,T]$.

\begin{lemma} \label{Contraction2}
Let $T<\infty$. The mapping $H_{T}=(F\circ F)_{T}$ has a unique fixed point.
\end{lemma}
\begin{proof}
By Lemmas \ref{L:MapToItself}, \ref{Contraction1} and the Banach fixed-point theorem we readily obtain that there is $0<T_{0}<\infty$ such that $H_{T_{0}}(m)$ will be a contraction map leading to (\ref{SDEY}) having a unique solution on $[0, T_{0}]$. We then extend this construction to the whole interval $[0,T]$ by dividing the interval $[0,T]$ into sub-intervals $[0, T_0], [T_0, 2 T_0], \ldots, [T- T_0, T]$. In each sub-interval, it can be shown that the solution is unique by proving a contraction as was done in Lemma \ref{Contraction1}, which can be done as $T_{0}$ can be always taken to be of the same magnitude, i.e. it does not depend on which sub-interval is being examined. This concludes the proof.
\end{proof}

\section{Proof of the Main Results} \label{FinalProof}
We now collect the results to prove Theorem \ref{TheoremLLN}, Corollary \ref{CorollaryLLN}, and Theorem \ref{TheoremChaos}.

\begin{proof}[Proof of Theorem \ref{TheoremLLN}]
Let $\pi^N$ be the probability measure corresponding to $\mu^N$. Each $\pi^N$ takes values in the set of probability measures $\mathcal{M} \big{(} D_E([0,T]) \big{)}$. Relative compactness, proven in Section \ref{RelativeCompactness}, implies that every subsequence $\pi^{N_k}$ has a further sub-sequence $\pi^{N_{k_m}}$ which weakly converges. Section \ref{Identification} proves that any limit point $\pi$ of $\pi^{N_{k_m}}$ will satisfy the evolution equation (\ref{EvolutionEquationIntroduction}). Section \ref{Uniqueness} proves that the solution of the evolution equation (\ref{EvolutionEquationIntroduction}) is unique. Therefore, by Prokhorov's Theorem, $\pi^N$ weakly converges to $\pi$, where $\pi$ is the distribution of $\bar \mu$, the unique solution of (\ref{EvolutionEquationIntroduction}). That is, $\mu^N$ converges in distribution to $\bar \mu$.
\end{proof}

\begin{proof}[Proof of Corollary \ref{CorollaryLLN}]
The result follows from applying integration by parts to (\ref{EvolutionEquationIntroduction}) using the assumption that $ p(t,c,w) \rightarrow 0$ as $|c|, \parallel w\parallel \rightarrow \infty$. We also note that if a solution exists to (\ref{fpPDE}), then it is unique due to the uniqueness of (\ref{EvolutionEquationIntroduction}).
\end{proof}

\begin{proof}[Proof of Theorem \ref{TheoremChaos}]
By Theorem \ref{TheoremLLN} we have that the scaled empirical measure $\mu^N_{\cdot}$
converges in distribution in $D_{E}([0,T])$ towards a deterministic limit $\bar{\mu}_{\cdot}$ with the joint distribution of the weights $ ( c^i_k, w^i_k)_{i=1}^N \in  ( \mathbb{R}^{1+d} )^{\otimes N}$ being exchangeable. Then, by the Tanaka-Sznitman theorem (see for example Theorem 3.2 in \cite{Gottlieb} or \cite{Sznitman}) we get that $\rho^N$ will be $\bar{\mu}$-chaotic.
\end{proof}

\section{Conclusion} \label{Conclusion}

In this paper we develop a law of large numbers result for neural networks with a single hidden layer as the number of hidden units and stochastic gradient descent iterations grow. The limiting distribution of the parameters is rigorously shown to satisfy an explicitly stated first-order nonlinear deterministic PDE, in the form of a measure evolution equation. The limiting PDE is a function of the inputs to the model, such as the learning rate, activation function, and distribution of the observed data. A numerical study on the well-known MNIST dataset illustrates the theoretical results of this paper. In related work which builds upon the results in this paper, a central limit theorem has been proven for single-layer neural networks in \cite{CLTneuralnetwork} and a law of large numbers has been proven for deep neural networks in \cite{DNN}. 

\newpage

\appendix

\section{Proof of (\ref{BoundUniqueProof0000AABB2})}  \label{AdditionalTechnicalDetailsAppendixB}

We will show that (\ref{BoundUniqueProof0000AABB}) can be bounded in terms of $D_{s,4}(m^1, m^2)$. For notational convenience, define $Z_s \vcentcolon =  |c_s^{2,'} - c_s^{1,'} |   +    \parallel w^{2,'}_s - w^{1,'}_s  \parallel$.

\begin{eqnarray}
\mathbb{E} \big{[} (1 + |c^{1,'}_s| )  Z_s \big{]}   &=&  \mathbb{E} \big{[} (1 + |c^{1,'}_s| )  Z_s  (\mathbf{1}_{Z_s \leq 1} + \mathbf{1}_{Z > 1} ) \big{]} \notag \\
&=& \mathbb{E} \bigg{[} (1 + |c^{1,'}_s| )    \big{(}  (Z_s \wedge 1) \mathbf{1}_{Z_s \leq 1} + Z_s (Z_s \wedge 1) \mathbf{1}_{Z_s > 1} \big{)} \bigg{]} \notag \\
&=&  \mathbb{E} \bigg{[} (1 + |c^{1,'}_s| )    \big{(}   \mathbf{1}_{Z_s \leq 1} + Z_s  \mathbf{1}_{Z_s > 1} \big{)} (Z_s \wedge 1) \bigg{]} \notag \\
&\leq&  \mathbb{E} \bigg{[} (1 + |c^{1,'}_s| )   (   1 + Z_s  ) (Z_s \wedge 1) \bigg{]} \notag \\
&\leq&  \mathbb{E} \bigg{[} (1 + |c^{1,'}_s| )  (  1 + |c_s^{2,'} | + |c_s^{1,'} | + \parallel w^{1,'}_s \parallel + \parallel w^{2,'}_s \parallel  )  (Z_s \wedge 1) \bigg{]} \notag \\
&\leq& C \bigg{[}  \mathbb{E}  (1 +  |c^{1,'}_s|^4 +|c_s^{2,'} |^4 +  \parallel w^{1,'}_s \parallel^4 + \parallel w^{2,'}_s \parallel^4  ) \bigg{]}^{1/2} \bigg{[} \mathbb{E}  (Z_s \wedge 1)^2 \bigg{]}^{1/2} \notag \\
&\leq& C  \bigg{[} \mathbb{E} \left( Z_s^4 \wedge 1\right) \bigg{]}^{1/4} \notag \\
&\leq& C  \bigg{[} \mathbb{E} \left( \big{(}  |c_s^{2,'} - c_s^{1,'} |^4  +    \parallel w^{2,'}_s - w^{1,'}_s  \parallel^4_4  \big{)}\wedge 1 \right) \bigg{]}^{1/4}.\nonumber
\end{eqnarray}
The sixth line uses the Cauchy-Schwartz inequality and Young's inequality. The seventh line uses the facts that $m^1$ and $m^2$ have bounded fourth order moments. The eighth line uses Young's inequality. We have also used the facts that $(z \wedge 1)^4 = z^4 \wedge 1$ and $(K z) \wedge 1 \leq K (z \wedge 1)$ when $z \geq 0$ and $K \geq 1$.

Therefore,
\begin{eqnarray}
&\phantom{.}& \bigg{|} \int_0^t  \int_{\mathcal{X}\times\mathcal{Y}}  \la c'_s \sigma (w'_s x) , m^{2}-m^{1}\ra   \sigma ( w_s^1\cdot x) \pi(dx,dy) ds \bigg{|} \notag \\
&\leq& C \int_0^t  \bigg{[}\mathbb{E} \left( \big{(}  |c_s^{2,'} - c_s^{1,'} |^4  +    \parallel w^{2,'}_s - w^{1,'}_s  \parallel^4_4  \big{)}\wedge 1\right) \bigg{]}^{1/4} ds \notag \\
&\leq& C  \int_0^t  \bigg{[} \mathbb{E} \left( \sup_{u \leq s} \big{(}  |c_u^{2,'} - c_u^{1,'} |^4  +    \parallel w^{2,'}_u - w^{1,'}_u  \parallel^4_4  \big{)}\wedge 1 \right)\bigg{]}^{1/4} ds.\nonumber
\end{eqnarray}

Since this inequality holds for any joint distribution $\gamma(m^1, m^2)$, we have that (\ref{BoundUniqueProof0000AABB2}) holds.

\section{Proof of (\ref{Eq:BoundForWt})} \label{AdditionalTechnicalDetails}

By definition, we have
\begin{align}
w_t^1  - w_t^2  &=   \int_0^t   \int_{\mathcal{X}\times\mathcal{Y}} \alpha \bigg{(} y - \la G_{s,x},m^1 \ra \bigg{)} c_s^1 \sigma' (  w_s^1 \cdot x) x \pi(dx,dy) ds \notag \\
&\quad-  \int_0^t   \int_{\mathcal{X}\times\mathcal{Y}} \alpha \bigg{(} y - \la G_{s,x},m^2 \ra \bigg{)} c_s^2 \sigma' (  w_s^2 \cdot x) x \pi(dx,dy) ds \notag \\
&= \alpha  \int_0^t   \int_{\mathcal{X}\times\mathcal{Y}} x \bigg{[} y  \big{(}  c_s^1 \sigma' (  w_s^1 \cdot x)  - c_s^2 \sigma' (  w_s^2 \cdot x)    \big{)} + \big{(} \la G_{s,x},m^2 \ra - \la G_{s,x},m^1 \ra \big{)} c_s^2 \sigma' (  w_s^2 \cdot x)  \notag \\
&\quad+  \la G_{s,x},m^1 \ra  \big{(} c_s^2 \sigma' (  w_s^2 \cdot x)  - c_s^1 \sigma' (  w_s^1 \cdot x)  \big{)}  \bigg{]} \pi(dx,dy)  ds \notag \\
&= \alpha  \int_0^t   \int_{\mathcal{X}\times\mathcal{Y}}  x \bigg{[} y \big{(} c_s^1 - c_s^2 \big{)} \sigma' (  w_s^2 \cdot x)  + y  c_s^1 \big{(}     \sigma' (  w_s^1 \cdot x)   - \sigma' (  w_s^2 \cdot x)  \big{)} \notag \\
&\quad+   \big{(} \la G_{s,x},m^2 \ra - \la G_{s,x},m^1 \ra \big{)} c_s^2 \sigma' (  w_s^2 \cdot x) + \la G_{s,x},m^1 \ra  \big{(} c_s^2 \sigma' (  w_s^2 \cdot x)  - c_s^1 \sigma' (  w_s^1 \cdot x)  \big{)}  \bigg{]} \pi(dx,dy)  ds \notag \\
&= \alpha  \int_0^t   \int_{\mathcal{X}\times\mathcal{Y}}  x \bigg{[} y \big{(} c_s^1 - c_s^2  \big{)} \sigma' (  w_s^2 \cdot x)  + y  c_s^1 \big{(}   \sigma' (  w_s^1 \cdot x) -  \sigma' (  w_s^2 \cdot x)   \big{)} \notag \\
&\quad+   \big{(} \la G_{s,x},m^2 \ra - \la G_{s,x},m^1 \ra \big{)} c_s^2 \sigma' (  w_s^2 \cdot x) \notag \\
&\quad+ \la G_{s,x},m^1 \ra  \big{(} c_s^2 - c_s^1 \big{)} \sigma' (  w_s^2 \cdot x) +  \la G_{s,x},m^1 \ra  c_s^1  \big{(}\sigma' (  w_s^2 \cdot x) -  \sigma' (  w_s^1 \cdot x)  \big{)}  \bigg{]} \pi(dx,dy)  ds.\nonumber
\end{align}

Therefore, we have the bound
\begin{align}
\parallel w_t^2 - w_t^1 \parallel &\leq C  \int_0^t   \int_{\mathcal{X}\times\mathcal{Y}}  \parallel x \parallel  \bigg{[} |y| \big{|} c_s^2 - c_s^1 \big{|}   + |y|  |c_s^1| \big{|}   \sigma' (  w_s^2 \cdot x)  -  \sigma' (  w_s^1 \cdot x)  \big{|} \notag \\
&\quad+   \big{|} \la G_{s,x},m^2 \ra - \la G_{s,x},m^1 \ra \big{|} | c_s^2|  \notag \\
&\quad+ | \la G_{s,x},m^1 \ra |  \big{|} c_s^2 - c_s^1 \big{|} +   | \la G_{s,x},m^1 \ra |   | c_s^1 |  \big{|}\sigma' (  w_s^2 \cdot x) -  \sigma' (  w_s^1 \cdot x)  \big{|} \bigg{]} \pi(dx,dy) ds.\nonumber
\end{align}
Due to $m^1$ having bounded moments, $| \la G_{s,x},m^1 \ra | < C$ (see Section \ref{Uniqueness} for similar calculations).  Due to Assumption \ref{A:Assumption1} on $\sigma$,
\begin{align}
\big{|}   \sigma' (  w_s^2 \cdot x)  -  \sigma' (  w_s^1 \cdot x)  \big{|} &\leq C | w_s^2 \cdot x  - w_s^1 \cdot x | = C \big{|} \sum_{i=1}^d ( w_s^{2,i} x_i - w_s^{1,i} x_i ) \big{|} \notag \\
&\leq C   \sum_{i=1}^d |x_i| \big{|} w_s^{2,i}  - w_s^{1,i}  \big{|} \leq C  \sum_{i=1}^d \norm{x} \big{|} w_s^{2,i}  - w_s^{1,i}  \big{|} = C  \norm{x} \parallel w_s^2 - w_s^1 \parallel.\nonumber
\end{align}
Using these inequalities and the bounded moments of $\pi(dx,dy)$, we can calculate the upper bound
\begin{align}
\parallel w_t^2 - w_t^1 \parallel &\leq C  \int_0^t  \bigg{[}  \big{|} c_s^2 - c_s^1 \big{|}   + |c_s^1|   \parallel w_s^2 - w_s^1 \parallel +   \big{|} \la G_{s,x},m^2 \ra - \la G_{s,x},m^1 \ra \big{|} | c_s^2|  \bigg{]} ds.\nonumber
\end{align}

Using the same approach as in the bound for  (\ref{BoundUniqueProof0000AABB}), see Appendix \ref{AdditionalTechnicalDetailsAppendixB}, we have the bound
\begin{eqnarray}
 \big{|} \la G_{s,x},m^2 \ra - \la G_{s,x},m^1 \ra \big{|} | c_s^2|  \leq C   D_{s,4}(m^1, m^2) | c_s^2|.\nonumber
\end{eqnarray}

Therefore, we have obtained (\ref{Eq:BoundForWt}).

\end{document}